\documentclass{siamltex}
\usepackage{amssymb,amsmath,amsfonts}
\usepackage[ruled,section]{algorithm}
\usepackage{graphicx}
\usepackage{algorithmic,verbatim,url,comment,color,subfigure}

\newcommand{\E}{\mathbb{E}}
\newcommand{\Prob}{\mathbb{P}}
\newcommand{\range}{\mathsf{range}}
\newcommand{\trace}{\mathsf{trace}}
\newcommand{\sr}{\mathsf{sr}}
\newcommand{\srr}{\mathsf{\rho}}

\newcommand{\norm}[1]{\left\|#1\right\|}
\newcommand{\bA}{\mathbf{A}}
\newcommand{\bV}{\mathbf{V}}

\newcommand{\bU}{\mathbf{U}}
\newcommand{\bQ}{\mathbf{Q}}
\newcommand{\bW}{\mathbf{W}}
\newcommand{\bP}{\mathbf{P}}
\newcommand{\bX}{\mathbf{X}}
\newcommand{\bY}{\mathbf{Y}}
\newcommand{\bZ}{\mathbf{Z}}
\newcommand{\bI}{\mathbf{I}}
\newcommand{\bB}{\mathbf{B}}
\newcommand{\bC}{\mathbf{C}}
\newcommand{\bH}{\mathbf{H}}
\newcommand{\bL}{\mathbf{L}}
\newcommand{\bS}{\mathbf{S}}
\newcommand{\bzero}{\mathbf{0}}
\newcommand{\bSigma}{\mathbf{\Sigma}}
\renewcommand{\bP}{\mathbf{P}}
\newcommand{\intdim}{\mathsf{intdim}}

\newtheorem{remark}[theorem]{Remark}

\newtheorem{example}[theorem]{Example}

\title{Randomized Approximation of the Gram Matrix: Exact Computation and Probabilistic 
Bounds\thanks{The first author was supported in part by 
Department of Education Grant P200A090081. The second author was supported 
in part by NSF grant CCF-1145383, and also
acknowledges the support from the XDATA Program of the Defense Advanced
Research Projects Agency (DARPA), administered through Air Force
Research Laboratory contract FA8750-12-C-0323 FA8750-12-C-0323.
This material was based upon work partially supported by the National
Science Foundation under Grant DMS-1127914 to the Statistical and
Applied Mathematical Sciences Institute.}
}
\author{
John T. Holodnak\thanks{Department of Mathematics, North Carolina State University, P.O. Box 8205, Raleigh, NC 27695-8205, USA, (\texttt{jtholodn@ncsu.edu}, \texttt{http://www4.ncsu.edu/{\char'176}jtholodn/})}
\and
Ilse C. F. Ipsen\thanks{%
Department of Mathematics, North Carolina State University, P.O. Box 8205,
Raleigh, NC 27695-8205, USA, (\texttt{ipsen@ncsu.edu}, 
\texttt{http://www4.ncsu.edu/{\char'176}ipsen/})}
}
\begin{document}
\maketitle

\begin{abstract} 
Given a real matrix $\bA$ with $n$ columns, the problem is to 
approximate the Gram product $\bA\bA^T$ by $c\ll n$ weighted outer products of
columns of $\bA$. Necessary and sufficient conditions for the exact computation
of $\bA\bA^T$ (in exact arithmetic) 
from $c\geq \rank(\bA)$ columns depend on the right
singular vector matrix of $\bA$. For a Monte-Carlo 
matrix multiplication algorithm by Drineas 
et al. that samples outer products,
we present probabilistic bounds for the 2-norm relative error due to
randomization. The
bounds depend on the stable rank or the rank of $\bA$, but not on
the matrix dimensions.
Numerical experiments illustrate that the bounds are
informative, even for stringent success probabilities and
matrices of small dimension.
We also derive bounds for the smallest singular value and 
the condition number of  matrices obtained by 
sampling rows from orthonormal matrices.
\end{abstract}


\begin{AM} 
68W20, 65C05, 15A18, 65F20, 65F35
\end{AM}

\section{Introduction}\label{s_into}
Given a real matrix 
$\bA=\begin{pmatrix} A_1 & \ldots & A_n\end{pmatrix}$ with $n$
columns $A_j$, can one approximate the Gram matrix $\bA\bA^T$
from just a \textit{few} columns? We answer this question
by presenting deterministic conditions for the
exact\footnote{We assume infinite precision, and no round off errors.}
computation of $\bA\bA^T$ from a few columns,
and probabilistic error bounds for approximations.

Our motivation (Section~\ref{s_mot}) is followed by an overview
of the results (Section~\ref{s_over}), and a literature survey
(Section~\ref{s_litsec}). Those not familiar with established
notation can find a review in Section~\ref{s_not}.

\subsection{Motivation}\label{s_mot}
The objective is the analysis of a randomized algorithm 
for approximating $\bA\bA^T$. Specifically, it is a Monte Carlo algorithm for
sampling outer products and represents a special case of the ground
breaking work on randomized matrix multiplication by Drineas, Kannan,
and Mahoney \cite{DK2001,DKM2006}.

The basic idea is to represent $\bA\bA^T$ as a sum of outer products
of columns,
$$\bA\bA^T = A_1A_1^T + \cdots + A_nA_n^T.$$
The Monte Carlo algorithm \cite{DK2001,DKM2006}, when provided with
a user-specified positive integer $c$, samples $c$ columns $A_{t_1}$, $\ldots$, $A_{t_c}$ according to probabilities $p_j$, $1 \leq j \leq n$, 
and then approximates $\bA\bA^T$ by a weighted sum of $c$ outer products
$$\bX=w_1A_{t_1}A_{t_1}^T +\cdots + w_cA_{t_c}A_{t_c}^T.$$
The weights are set to $w_j = 1/(cp_{t_j})$ so that $\bX$ is an
unbiased estimator, $\E[\bX]=\bA\bA^T$.
Intuitively, one would expect the algorithm to do well for matrices of
low rank.  

The intuition is based on the singular value
decomposition. Given left singular vectors $U_j$ associated with
the $k\equiv\rank(\bA)$ non-zero singular values $\sigma_j$ of $\bA$, 
one can represent $\bA\bA^T$ as a sum of $k$ outer products,
$$\bA\bA^T = \sigma_1^2\, U_1U_1^T+\cdots + \sigma_k^2\,U_kU_k^T.$$
Hence for matrices $\bA$ of low rank, a few left singular vectors
and singular values suffice to reproduce $\bA\bA^T$ exactly.  
Thus, if $\bA$ has columns that ``resemble'' its left singular vectors, 
the Monte Carlo algorithm should have a chance to perform well.

\subsection{Contributions and Overview}\label{s_over}
We sketch the main contributions of this paper.
All proofs are relegated to Section~\ref{s_proofs}.

\subsubsection{Deterministic conditions for exact computation 
(Section~\ref{s_det})}
To calibrate the potential of the Monte-Carlo algorithm 
\cite{DK2001,DKM2006} and establish connections to 
existing work in linear algebra, we
first derive deterministic conditions that characterize when
$\bA\bA^T$ can be computed \textit{exactly} from
a few columns of $\bA$. Specifically:

\begin{itemize}
\item We present necessary and sufficient conditions (Theorem~\ref{t_exactg})
for computing $\bA\bA^T$ exactly 
from $c \geq \rank(\bA)$ columns $A_{t_1},\ldots,A_{t_c}$ of $\bA$,
$$\bA\bA^T=w_1\,A_{t_1}A_{t_1}^T +\cdots +w_c\,A_{t_c}A_{t_c}^T.$$
The conditions and weights $w_j$ depend on the right singular vector matrix $\bV$ 
associated with the non-zero singular values of $\bA$.
\item For matrices with $\rank(\bA)=1$, this is always possible 
(Corollary~\ref{c_exactg}).

\item In the special case where $c=\rank(\bA)$ (Theorem \ref{t_exact}),
the weights are equal to inverse leverage scores, $w_j=1/\|\bV^Te_{t_j}\|_2^2$.
However, they  do not necessarily correspond to the largest leverage scores.
\end{itemize}

\subsubsection{Sampling probabilities for the Monte-Carlo algorithm
(Section~\ref{s_randmm})}
Given an approximation $\bX$ from the Monte-Carlo 
algorithm \cite{DK2001,DKM2006}, we
are interested in the two-norm relative error due to randomization,
$\|\bX-\bA\bA^T\|_2/\|\bA\bA^T\|_2$.
Numerical experiments 
compare two types of sampling probabilities:
\begin{itemize}
\item  ``Optimal'' probabilities $p_j^{opt}=\|A_j\|_2^2/\|\bA\|_F^2$ 
\cite{DKM2006}, and 
\item Leverage score probabilities $p_j^{lev}=\|\bV^Te_j\|_2^2/k$
\cite{BDMI11,BMD09}.
\end{itemize}
The experiments illustrate that sampling columns
of $\bX$ with the ``optimal'' probabilities 
produces a smaller error than sampling with leverage score probabilities. 
This was not obvious a priori,
because the ``optimal'' probabilities are designed to minimize
the expected value of the Frobenius norm absolute error,
$\E[\|\bX-\bA\bA^T\|_F^2]$. Furthermore,
corresponding probabilites $p_j^{opt}$ and $p_j^{lev}$ 
can differ by orders of magnitude.

For matrices $\bA$ of rank one though, we show (Theorem~\ref{t_rankone})
that the probabilities are identical,
$p_j^{opt}=p_j^{lev}$ for $1\leq j\leq n$,
 and that the Monte Carlo algorithm always
produces the exact result, $\bX=\bA\bA^T$, when it samples with these 
probabilities.

\subsubsection{Probabilistic bounds (Sections \ref{s_error} and~\ref{s_error2})}
We present probabilistic bounds for $\|\bX-\bA\bA^T\|_2/\|\bA\bA^T\|_2$ 
when the Monte-Carlo algorithm
samples with two types of sampling probabilities.

\begin{itemize}
\item Sampling with ``nearly optimal'' probabilities 
$p_j^{\beta}\geq \beta \,p_j^{opt}$, where $\beta\leq 1$ 
(Theorems \ref{t_troppmult} and~\ref{t_minsker}).
We show that
$$\|\bX-\bA\bA^T\|_2/\|\bA\bA^T\|_2\leq \epsilon \quad
\mbox{with probability at least~} 1-\delta,$$
provided the number of sampled columns is at least 
$$c \geq c_0(\epsilon)\>\frac{\ln(\rho(\bA)/\delta)}{\beta \epsilon^2}\sr(\bA), 
\qquad \mbox{where} \quad 2\leq c_0(\epsilon) \leq 2.7.$$
Here $\srr(\bA)=\rank(\bA)$ or $\srr(\bA)=4\,\sr(\bA)$, where
$\sr(\bA)$ is the stable rank of~$\bA$. The bound containing $\rank(\bA)$
is tighter for matrices with $\rank(\bA)\leq 4\,\sr(\bA)$.

Note that the amount of sampling depends on the rank or the 
stable rank, but not on the dimensions of $\bA$. Numerical experiments  
(Section~\ref{s_cb}) illustrate that the bounds are
informative, even for stringent success probabilities and
matrices of small dimension.
\item Sampling with leverage score probabilities $p_j^{lev}$
(Theorem~\ref{t_levremove}).
The  bound corroborates the numerical experiments in Section~\ref{s_pcomp},
but is not as tight as the bounds for 
``nearly optimal'' probabilities, since it
depends only on $\rank(\bA)$, and $\rank(\bA)\geq \sr(\bA)$.
\end{itemize}

\subsubsection{Singular value bounds (Section~\ref{s_sing})}
Given a $m\times n$ matrix $\bQ$ with orthonormal rows,
$\bQ\bQ^T=\bI_m$, the Monte-Carlo algorithm 
computes $\bQ\bS$ by sampling $c\geq m$
columns from $\bQ$ with the ``optimal'' probabilities.
The goal is to 
derive a positive lower bound for the smallest singular value
$\sigma_m(\bQ\bS)$, as well as an upper
bound for the two-norm condition number with respect to left inversion
$\kappa(\bQ\bS)\equiv \sigma_1(\bQ\bS)/\sigma_m(\bQ\bS)$. 

Surprisingly, Theorem~\ref{t_troppmult} leads to bounds 
(Theorems \ref{t_singmatmult} and~\ref{t_condmatmult})
that are not always as tight as the ones below. 
These bounds are based on a Chernoff inequality and
represent a slight improvement over existing results.

\begin{itemize} 
\item Bound for the smallest singular value (Theorem~\ref{t_singchernoff}).
We show that 
$$\sigma_m\left(\bQ\bS \right) \geq \sqrt{1-\epsilon}\quad
\mbox{with probability at least~} 1-\delta,$$
provided the number of sampled  columns is at least
$$c \geq c_1(\epsilon)  \>m\>\frac{\ln(m/\delta)}{\epsilon^2}, \qquad 
\mbox{where} \quad 1\leq c_1(\epsilon)\leq 2.$$

\item Condition number bound (Theorem~\ref{t_condchernoff}).
We show that 
$$\kappa(\bQ\bS) \leq \frac{\sqrt{1+\epsilon}}{\sqrt{1-\epsilon}} \quad
\mbox{with probability at least~} 1-\delta,$$ 
provided the number of sampled columns is at least
$$c \geq c_2(\epsilon)  \>m\>\frac{\ln(2m/\delta)}{\epsilon^2},
\qquad \mbox{where} \quad 2\leq c_2(\epsilon)\leq 2.6.$$
\end{itemize}

In addition, we derive corresponding bounds for uniform sampling with and 
without replacement (Theorems \ref{t_singchernoff} and \ref{t_condchernoff}).

\subsection{Literature Review}\label{s_litsec}
We review bounds for the relative error due to randomization 
of general Gram matrix approximations
$\bA\bA^T$, and also for the smallest singular value and 
condition number of sampled matrices $\bQ\bS$ when $\bQ$ has orthonormal rows.

In addition to \cite{DK2001,DKM2006},
several other randomized matrix multiplication algorithms have been proposed 
\cite{BW2008,CL1997,CL1999,Liberty2013,Pagh2013,Sarlos2006}.  Sarl\'{o}s's
algorithms \cite{Sarlos2006} are based on matrix transformations.  Cohen
and Lewis \cite{CL1997,CL1999} approximate large elements of a matrix
product with a random walk  algorithm.  The algorithm
by Belabbas and Wolfe \cite{BW2008} is related to the Monte Carlo
algorithm \cite{DK2001,DKM2006}, but with different
sampling methods and weights. 
A second algorithm by Drineas et al. \cite{DKM2006} 
relies on matrix sparsification, and a third algorithm
\cite{DK2001} estimates each matrix element independently.  
Pagh \cite{Pagh2013} targets sparse matrices, while  Liberty
\cite{Liberty2013} estimates the Gram matrix $\bA\bA^T$ by iteratively
removing ``unimportant'' columns from~$\bA$.

Eriksson-Bique et al. \cite{ESS2011}
derive an importance sampling strategy that minimizes
the variance of the inner products computed by the Monte Carlo method.
Madrid, Guerra, and Rojas \cite{MGR2012} present experimental comparisons of 
different sampling strategies for specific classes of matrices.

Excellent surveys of randomized matrix algorithms in general
are given by Halko, Martinsson, and Tropp \cite{HMT2011}, and by
Mahoney \cite{Mahoney2011}.

\subsubsection{Gram matrix approximations}
We review existing bounds for the error due to randomization of the Monte Carlo algorithm  \cite{DK2001,DKM2006} for approximating  $\bA\bA^T$, where $\bA$ is a real $m \times n$ matrix.
Relative error bounds 
$\|\bX-\bA\bA^T\|/\|\bA\bA^T\|$ in the Frobenius norm and the two-norm
are summarized in Tables \ref{tt_fro} and~\ref{tt_two}.

Table~\ref{tt_fro}
shows probabilistic lower bounds for the number of sampled
columns so that the Frobenius norm relative error 
$\norm{\bX-\bA\bA^T}_F/\norm{\bA\bA^T}_F\leq \epsilon$.
Not listed is a bound for uniform sampling without 
replacement \cite[Corollary 1]{KMT2009}, 
because it cannot easily be converted to the
format of the other bounds, and a bound
for a greedy sampling strategy \cite[p. 5]{BW2008}.

\begin{table}
\begin{center}
\begin{tabular}{|c|c|l|}
\hline
Bound for \# samples & Sampling  & \multicolumn{1}{|c|}{Reference} \\
\hline
 $ \frac{(1+\sqrt{8\ln(1/\delta)})^2}{\epsilon^2}\frac{\norm{A}_F^4}{\norm{AA^T}_F^2} $ & opt &\cite[Theorem 2]{DKM2006} \\
\hline
 $ \frac{1}{\epsilon^2\delta}\frac{\norm{A}_F^4}{\norm{AA^T}_F^2} $ & opt & \cite[Lemma 1]{FKV1998}, \cite[Lemma 2]{FKV2004} \\
\hline
 $ \frac{n^2}{(n-1)\delta \epsilon^2} \frac{\sum_{j=1}^n{\norm{A_j}_2^4}}{\norm{AA^T}_F^2} $ & u-wor & \cite[Lemma 7]{DK2001} \\
\hline
$ \frac{36n\ln(1/\delta)}{\epsilon^2}\frac{\sum_{j=1}^n{\norm{A_i}_2^4}}{\norm{AA^T}_F^2}$ & u-wor &  \cite[Lemma 4.13]{BG2013}, \cite[Lemma 4.3]{Gittens2013} \\
\hline
\end{tabular}
\smallskip

\end{center}
\caption{Frobenius-norm error due to randomization: Lower 
bounds on the number $c$ of sampled columns in $\bX$,
so that $\|\bX-\bA\bA^T\|_F/\|\bA\bA^T\|_F\leq \epsilon$ with probability at 
least $1-\delta$. The second column specifies
the sampling strategy:  ``opt'' for sampling with ``optimal''
probabilities, and ``u-wor'' for uniform sampling without replacement.
The last two bounds are special cases of bounds for general
matrix products $\bA\bB$.}
\label{tt_fro}
\end{table}

Table~\ref{tt_two} shows probabilistic lower bounds for the number of sampled
columns so that the two-norm relative error 
$\norm{\bX-\bA\bA^T}_2/\norm{\bA\bA^T}_2\leq \epsilon$.
These bounds imply, roughly, that the number of sampled columns should be
at least $\Omega(\sr(\bA)\>\ln(\sr(\bA))$ or $\Omega(\sr(\bA)\>\ln(m))$.

\begin{table}
\begin{center}
\begin{tabular}{|c|l|}
\hline
Bound for \# samples & \multicolumn{1}{|c|}{Reference} \\
\hline
$ C\frac{\sr(A)}{\epsilon^2 \delta}\ln(\sr(A)/(\epsilon^2\delta))$  & \cite[Theorems 1.1 and 3.1, and their proofs]{RV2007} \\
\hline
$ \frac{4\sr(A)}{\epsilon^2}\ln(2m/\delta)$ & \cite[Theorem 17]{Magdon2011}, \cite[Theorem 20]{Magdon2010} \\
\hline
 $ \frac{96\sr(A)}{\epsilon^2}\ln\left( \frac{96\sr(A)}{\epsilon^2\sqrt{\delta}} \right)$ & \cite[Theorem 4]{Drineas2010}  \\
\hline
$ \frac{20\sr(A)}{\epsilon^2}\ln( 16\sr(A)/\delta)$ & 
\cite[Theorem~3.1]{MZ2011}, \cite[Theorem 2.1]{Zouzias2013} \\
\hline
$ \frac{21(1 + \sr(A))}{4\epsilon^2}\ln(4\sr(A)/\delta)$  & \cite[Example 4.3]{Hsu2012} \\
\hline
$ \frac{8m}{\epsilon^2}\ln(m/\delta)$ & \cite[Theorem 3.9]{Srivastava2010} \\
\hline
\end{tabular}
\smallskip

\end{center}
\caption{Two-norm error due to randomization, for sampling with 
``optimal'' probabilities: Lower bounds on the number $c$ of sampled columns
in $\bX$,
so that $\|\bX-\bA\bA^T\|_2/\|\bA\bA^T\|_2\leq \epsilon$ with probability at 
least $1-\delta$ for all bounds but the first.
The first bound contains an unspecified constant~$C$
and holds with probability at least $1-2\exp(\tilde{C}/\delta)$,
where $\tilde{C}$ is another unspecified constant
(our $\epsilon$ corresponds to $\epsilon^2/2$ in \cite[Theorem 1.1]{RV2007}).
The penultimate bound is a special case
of a bound for general matrix products $\bA\bB$, while the last 
bound applies only to matrices with orthonormal rows.}
\label{tt_two}
\end{table}

\begin{table}
\begin{center}
\begin{tabular}{|c|c|c|}
\hline
Bound for \#  samples & Sampling  & Reference \\
\hline
 $ \frac{6n \mu}{\epsilon^2}\ln(m/\delta)$ & u-wor & \cite[Lemma 4.3]{BG2013} \\
\hline
 $ \frac{4m}{\epsilon^2}\ln(2m/\delta)$ & opt & \cite[Lemma 13]{Boutsidis2011}  \\
\hline
 $\frac{3n\mu}{\epsilon^2}\ln(m/\delta)$ & u-wr, u-wor &  \cite[Corollary 4.2]{Ipsen2012} \\
\hline
 $ \frac{8n\mu}{3\epsilon^2}\ln(m/\delta)$ & u-wr &  \cite[Lemma 4.4]{BG2013} \\
\hline
 $\frac{2n\mu}{\epsilon^2}\ln(m/\delta)$ & u-wor & \cite[Lemma 1]{Git11} \\
\hline
\end{tabular}
\smallskip

\end{center}
\caption{Smallest singular value of a matrix~$\bQ\bS$ whose columns are sampled
from a $m\times n$ matrix~$\bQ$ with orthonormal rows:
Lower bounds on the number $c$ of sampled columns,
so that  $\sigma_m(\bQ\bS)\geq \sqrt{1-\epsilon}$ with probability at 
least $1-\delta$.
The second column specifies the sampling strategy: ``opt'' 
for sampling with ``optimal''  probabilities, 
``u-wr'' for uniform sampling with replacement, and ``u-wor'' 
for uniform sampling without replacement.}
\label{tt_sing}
\end{table}

\subsubsection{Singular value bounds}
We review existing bounds for the smallest singular value 
of a sampled matrix $\bQ\bS$, where 
$\bQ$ is $m\times n$ with orthonormal rows.  

Table~\ref{tt_sing} shows  
probabilistic lower bounds for the number of sampled 
columns so that the smallest singular value
$\sigma_m(\bQ\bS)\geq \sqrt{1-\epsilon}$. All bounds but one
contain the coherence~$\mu$.
Not  listed is a bound \cite[Lemma 4]{Drineas2010} that requires 
specific choices of $\epsilon$, $\delta$, and $\mu$.

\subsubsection{Condition number bounds}
We are aware of only two existing bounds for the two-norm condition number 
$\kappa(\bQ\bS)$
of a  matrix $\bQ\bS$ whose columns are sampled from a $m\times n$
matrix~$\bQ$ with orthonormal rows.
The first bound \cite[Theorem 3.2]{Avron2010} lacks explicit constants, while
the second one \cite[Corollary 4.2]{Ipsen2012} applies to uniform
sampling with and without replacement. It ensures 
$\kappa(\bQ\bS) \leq \tfrac{\sqrt{1+\epsilon}}{\sqrt{1-\epsilon}}$
with probability at least $1-\delta$, provided 
the number of sampled columns in $\bQ\bS$ is at least
$c \geq 3\>n\mu\>\ln(2m/\delta)/\epsilon^2$.

\subsubsection{Relation to subset selection}
The Monte Carlo algorithm selects outer products from 
$\bA\bA^T$, which is equivalent to selecting columns from $\bA$, 
hence it can be viewed as a form of randomized
column subset selection. 

The traditional deterministic subset selection methods 
select exactly the required number of columns,
by means of rank-revealing QR decompositions or SVDs
\cite{ChI91a,Golub1965,GoKS76,GI1996,HoPa90}. In contrast, more recent methods
are motivated by applications to graph sparsification
\cite{BSS2009,BSS2012,Srivastava2010}. 
They oversample columns from a matrix $\bQ$ with orthonormal rows,
by relying on a  \textit{barrier sampling} strategy\footnote{The name
comes about as follows: Adding a column $q$ to $\bQ\bS$ amounts to a 
rank-one update $qq^T$ for the Gram matrix 
$(\bQ\bS)\>(\bQ\bS)^T$. The eigenvalues
of this matrix, due to interlacing, form ``barriers'' for the eigenvalues 
of the updated matrix $(\bQ\bS)\>(\bQ\bS)^T+qq^T$.}.
The accuracy of the selected columns $\bQ\bS$ is determined by bounding the
\textit{reconstruction error}, which  views $(\bQ\bS)\>(\bQ\bS)^T$
as an approximation to $\bQ\bQ^T=I$
\cite[Theorem 3.1]{BSS2009}, \cite[Theorem 3.1]{BSS2012},
\cite[Theorem 3.2]{Srivastava2010}. 

Boutsidis \cite{Boutsidis2011} extends this work to general Gram
matrices $\bA\bA^T$. Following \cite{GoKS76}, he selects columns
from the right singular
vector matrix $\bV^T$ of $\bA$, and applies barrier sampling simultaneously 
to the dominant and subdominant subspaces of $\bV^T$.

In terms of randomized algorithms for subset selection, the two-stage
algorithm by Boutsidis et al. \cite{BMD09} samples columns in the
first stage, and performs a deterministic subset selection on the
sampled columns in the second stage.  Other approaches include volume
sampling \cite{FKV1998,FKV2004}, and CUR decompositions
\cite{Drineas2008}.

\subsubsection{Leverage scores} In the late seventies, statisticians
introduced leverage scores for outlier detection in regression
problems \cite{ChatterH86,HoagW78,VelleW81}. More recently,
Drineas, Mahoney et al.
have pioneered the use of leverage scores for importance
sampling in randomized algorithms, such as 
CUR decompositions \cite{Drineas2008}, least squares problems
\cite{Drineas2006}, and column subset selection \cite{BMD09}, 
see also the perspectives on statistical leverage \cite[\S 6]{Mahoney2011}. 
Fast approximation algorithms are being designed to make
the computation of leverage scores more affordable 
\cite{DMMW2012,LMP2013,Magdon2010}.


\subsection{Notation}\label{s_not}
All matrices are real. Matrices that can have more than one column are 
indicated in bold face, and column vectors and scalars in italics. 
The columns of the $m\times n$ matrix $\bA$ are denoted by 
$\bA=\begin{pmatrix} A_1 & \cdots & A_n\end{pmatrix}$.
The  $n\times n$ identity matrix is 
$\bI_n\equiv\begin{pmatrix} e_1 & \cdots & e_n\end{pmatrix}$,
whose columns are the canonical vectors $e_j$.

The thin Singular Value Decomposition (SVD) of a $m\times n$ matrix $\bA$ 
with $\rank(\bA)=k$ is
$\bA=\bU\bSigma \bV^T$, where the $m\times k$ matrix $\bU$ and
the $n\times k$ matrix $\bV$ have orthonormal columns, 
$\bU^T\bU=\bI_k=\bV^T\bV$, and the $k\times k$ 
diagonal matrix of singular values is
$\bSigma=\diag\begin{pmatrix}\sigma_1 &\ldots & \sigma_k\end{pmatrix}$,
with $\sigma_1\geq \cdots\geq \sigma_k>0$.  
The Moore-Penrose inverse of~$\bA$ is 
$\bA^{\dagger}\equiv \bV\bSigma^{-1}\bU^T$.
The unique symmetric positive semi-definite square  root of a 
symmetric positive semi-definite matrix $\bW$ is denoted by $\bW^{1/2}$.

The norms
in this paper are the two-norm $\norm{\bA}_2 \equiv \sigma_1$, and the 
Frobenius norm 
$$\norm{\bA}_F \equiv \sqrt{\sum_{j=1}^n{\|A_j\|_2^2}} = 
\sqrt{\sigma_1^2 +  \cdots  + \sigma_k^2}.$$
The \textit{stable rank} of a non-zero matrix $\bA$ is 
$\sr(\bA)\equiv\|\bA\|_F^2/\|\bA\|_2^2$, where 
$1\leq \sr(\bA)\leq \rank(\bA)$.

Given  a $m\times n$ matrix 
$\bQ=\begin{pmatrix}Q_1 &\cdots & Q_n\end{pmatrix}$
with orthonormal rows, $\bQ\bQ^T=\bI_m$,
the \textit{two-norm condition number} with regard to left inversion
is $\kappa(\bQ)\equiv \sigma_1(\bQ)/\sigma_m(\bQ)$;
the \textit{leverage scores} \cite{Drineas2006, Drineas2008,Mahoney2011}
are the squared columns norms
$\|Q_j\|_2^2$, $1\leq j\leq m$;
and the \textit{coherence} \cite{Avron2010,CaR09}
is the largest leverage score,
\begin{eqnarray*}
\mu\equiv \max_{1\leq j\leq m}{\|Q_j\|_2^2}.
\end{eqnarray*}

The expected value of a scalar or a matrix-valued
random random variable $\bX$ is $\E[\bX]$; and the
probability of an event $\mathcal{X}$ is $\Prob[\mathcal{X}]$.

\section{Deterministic conditions for exact computation}\label{s_det}
To gauge the potential of the Monte Carlo algorithm, 
and to establish a connection to existing work in linear algebra, we 
first consider the best case: The \textit{exact} computation of $\bA\bA^T$ 
from a few columns. That is: 
Given $c$ not necessarily distinct columns $A_{t_1},\ldots,A_{t_c}$,
under which conditions is 
$w_1A_{t_1}A_{t_1}^T +\cdots + w_cA_{t_c}A_{t_c}^T=\bA\bA^T$?

Since a column can be selected more than once, and therefore
the selected columns may not form a submatrix of $\bA$,
we express the $c$ selected columns as $\bA\bS$, where $\bS$ is a 
$n\times c$ sampling matrix with
$$\bS=\begin{pmatrix}e_{t_1} & \ldots & e_{t_c}\end{pmatrix}, \qquad
1\leq t_1\leq \ldots\leq t_c\leq n.$$
Then one can write
$$w_1A_{t_1}A_{t_1}^T +\cdots + w_cA_{t_c}A_{t_c}^T=(\bA\bS) \bW (\bA\bS)^T,$$ 
where $\bW=\diag\begin{pmatrix}w_1 & \cdots & w_c\end{pmatrix}$
is diagonal weighting matrix. We answer two questions in this section:

\begin{enumerate}
\item Given a set of $c$ columns $\bA\bS$ of $\bA$,
when is $\bA\bA^T=(\bA\bS)\,\bW\, (\bA\bS)^T$
\textit{without any constraints} on $\bW$? The answer is an
expression for a matrix $\bW$ with minimal Frobenius norm
(Section~\ref{s_opt}).

\item Given a set of $c$ columns $\bA\bS$ of $\bA$,
what are necessary and sufficient conditions under which
$(\bA\bS) \bW (\bA\bS)^T=\bA\bA^T$ for a \textit{diagonal matrix $\bW$}?
The answer depends on the right singular vector matrix of $\bA$
(Section~\ref{s_exact}).
\end{enumerate}

\subsection{Optimal approximation (no constraints on~$\bW$)}\label{s_opt}
For a given set of $c$ columns $\bA\bS$ of $\bA$, 
we determine a matrix $\bW$ of minimal Frobenius norm that  minimizes 
the absolute error of $(\bA\bS) \bW (\bA\bS)^T$ in the Frobenius norm.

The following is a special case of \cite[Theorem 2.1]{FrT07}, without
any constraints on the number of columns in 
$\bA\bS$. The idea is to represent $\bA\bS$
in terms of the thin SVD of $\bA$ as $\bA\bS=\bU\bSigma (\bV^T\bS)$.

\begin{theorem}\label{t_lr}
Given $c$ columns $\bA\bS$ of $\bA$, not necessarily distinct, 
the unique solution of
$$\min_{\bW}{\|\bA\bA^T-(\bA\bS) \bW (\bA\bS)^T\|_F}$$
with minimal Frobenius norm is 
$\bW_{opt}=(\bA\bS)^{\dagger}\>\bA\bA^T\>((\bA\bS)^{\dagger})^T$.

If, in addition, $\rank(\bA\bS)=\rank(\bA)$, then 
$$(\bA\bS)\bW_{opt}(\bA\bS)^T=\bA\bA^T \qquad and \qquad
\bW_{opt}=(\bV^T\bS)^{\dagger}((\bV^T\bS)^{\dagger})^T.$$

If also $c=\rank(\bA\bS)=\rank(\bA)$, then 
$$(\bA\bS)\bW_{opt}(\bA\bS)^T=\bA\bA^T \qquad and \qquad
\bW_{opt}=(\bV^T\bS)^{-1}(\bV^T\bS)^{-T}.$$
\end{theorem}

\begin{proof} See Section~\ref{s_plr}.
\end{proof} 

Theorem~\ref{t_lr} implies that if $\bA\bS$ has maximal rank, 
then the solution $\bW_{opt}$ of minimal Frobenius norm depends only 
on the right singular vector matrix of $\bA$ and in particular 
only on those columns $\bV^T\bS$ that correspond to the columns in $\bA\bS$.

\subsection{Exact computation with outer products (diagonal 
$\bW$)}\label{s_exact}
We present necessary and sufficient conditions under which
$(\bA\bS)\bW(\bA\bS)^T=\bA\bA^T$ for a 
non-negative diagonal matrix $\bW$, that is
$w_1A_{t_1}A_{t_1}^T +\cdots + w_cA_{t_c}A_{t_c}^T=\bA\bA^T$.

\begin{theorem}\label{t_exactg}
Let $\bA$ be a $m\times n$ matrix, and let $c\geq k\equiv\rank(\bA)$.
Then 
$$\sum_{j=1}^c{w_j\>A_{t_j}A_{t_j}^T}=\bA\bA^T$$
for weights $w_j\geq 0$, if and only if the $c\times k$ matrix 
$\bV^T\begin{pmatrix}\sqrt{w_1}\,e_{t_1} & \cdots & \sqrt{w_c}\, e_{t_c}
\end{pmatrix}$
has orthonormal rows.
\end{theorem}

\begin{proof} See Section~\ref{s_pexactg}.
\end{proof}

If $\bA$ has rank one, then any $c$ non-zero columns of $\bA$ will
do for representing $\bA\bA^T$, and explicit expressions for
the weights can be derived.

\begin{corollary}\label{c_exactg}
If $\rank(\bA)=1$ then for any $c$ columns $A_{t_j}\neq 0$,
$$\sum_{j=1}^c{w_j\>A_{t_j}A_{t_j}^T}=\bA\bA^T \qquad where \quad
w_j=\frac{1}{c\,\|\bV^Te_{t_j}\|_2^2}
=\frac{\|\bA\|_F^2}{\|A_{t_j}\|_2^2}, \quad 1\leq j\leq c.$$
\end{corollary}

\begin{proof} See Section~\ref{s_pcexactg}.
\end{proof}

Hence, in the special case of rank-one matrices, the weights are inverse
leverage scores of $\bV^T$ as well as inverse normalized column norms of $\bA$.
Furthermore,  in the special case $c=1$, 
Corollary~\ref{c_exactg} implies that
any non-zero column of $\bA$ can be chosen. In particular, choosing the 
column $A_l$ of largest norm yields a weight 
$w_1=1/\|\bV^Te_l\|_2^2$ of minimal value, where 
$\|\bV^Te_l\|_2^2$ is the coherence of~$\bV^T$.

In the following, we look at Theorem~\ref{t_exactg} in more detail, 
and distinguish the two cases when the
number of selected columns is greater than $\rank(\bA)$, and 
when it is equal to $\rank(\bA)$.

\subsubsection{Number of selected columns greater than $\rank(\bA)$}
We illustrate  the conditions of Theorem~\ref{t_exactg} when $c>\rank(\bA)$.
In this case, indices do not necessarily have to be distinct, 
and a column can occur repeatedly. 

\begin{example}
Let
$$\bV^T=\begin{pmatrix} 1&0&0&0\\ 0&1&0&0\end{pmatrix}$$
so that $\rank(\bA)=2$. Also let
$c=3$, and select the first column twice,
$t_1=t_2=1$ and $t_3=2$, so that
$$\bV^T\begin{pmatrix}e_1&e_1 &e_2\end{pmatrix}=
\begin{pmatrix} 1&1&0\\ 0&0&1\end{pmatrix}.$$
The weights $w_1=w_2=1/2$ and $w_3=1$ give a matrix
$$\bV^T\begin{pmatrix}2^{-1/2}e_1 &2^{-1/2} e_1 &e_2\end{pmatrix}
=\begin{pmatrix} 2^{-1/2}& 2^{-1/2} &0\\ 0&0&1\end{pmatrix}$$
with orthonormal rows. Thus, an exact representation does not require
distinct indices.
\end{example}

However, although the above weights yield an exact representation, the
corresponding weight matrix does not have minimal Frobenius norm.

\begin{remark}[Connection to Theorem~\ref{t_lr}]
If $c>k\equiv\rank(\bA)$ in Theorem~\ref{t_exactg}, 
then no diagonal weight matrix 
$\bW=\diag\begin{pmatrix}w_1 & \cdots & w_c\end{pmatrix}$
can be a minimal norm solution $\bW_{opt}$ in  Theorem~\ref{t_lr}. 

To see this, note that for $c>k$, the
columns $A_{t_1},\ldots,A_{t_c}$ are linearly dependent. Hence the
$c\times c$ minimal Frobenius norm solution $\bW_{opt}$
has rank equal to~$k<c$. If $\bW_{opt}$ were to be diagonal, it could have
only $k$ non-zero diagonal elements, hence 
the number of outer products would be $k<c$, a contradiction.

To illustrate this, let 
$$\bV^T =\frac{1}{\sqrt{2}}\>\begin{pmatrix} 1& 0 & 1&0\\ 0&1&0&1
\end{pmatrix}$$
so that $\rank(\bA)=2$. Also, let
$c=3$, and select columns $t_1=1$, $t_2=2$ and $t_3=3$, so that 
$$\bV^T\bS\equiv \bV^T\begin{pmatrix}e_1&e_2&e_3\end{pmatrix}=\frac{1}{\sqrt{2}}\>\begin{pmatrix} 1& 0 & 1\\ 0&1&0 \end{pmatrix}.$$
Theorem~\ref{t_lr} implies that the solution with minimal Frobenius norm is
$$\bW_{opt} =(\bV^T\bS)^{\dagger}((\bV\bS^T)^{\dagger}) =
\begin{pmatrix} 1/2 & 0 & 1/2\\ 0&2&0\\ 1/2 & 0 & 1/2\end{pmatrix},$$
which is not diagonal.

However
$\bW=\diag\begin{pmatrix}1 & 2 & 1\end{pmatrix}$ is also a solution
since $\bV^T\bS\bW^{1/2}$ has orthonormal rows. But $\bW$ does not have
minimal Frobenius norm since $\|\bW\|_F^2=6$, while
$\|\bW_{opt}\|_F^2=5$.
\end{remark}

\subsubsection{Number of selected columns equal to $\rank(\bA)$}
If $c=\rank(\bA)$, then no
column of $\bA$ can be selected more than once, hence the selected 
columns form a submatrix of $\bA$. 
In this case Theorem~\ref{t_exactg} can be strengthened:
As for the rank-one case in Corollary~\ref{c_exactg}, an explicit
expression for the weights in terms of leverage scores can be derived.

\begin{theorem}\label{t_exact}
Let $\bA$ be a $m\times n$ matrix with $k\equiv\rank(\bA)$.
In addition to the conclusions of Theorem~\ref{t_exactg} the following 
also holds: If 
$$\bV^T\begin{pmatrix}\sqrt{w_1}e_{t_1} & \cdots & \sqrt{w_k}e_{t_k}\end{pmatrix}$$
has orthonormal rows, then it is an orthogonal matrix, 
and $w_j=1/\|\bV^Te_{t_j}\|_2^2$, $1\leq j\leq k$.
\end{theorem}

\begin{proof} See Section~\ref{s_pexact}. \end{proof}

Note that the columns selected from $\bV^T$ do not necessarily 
correspond to the largest leverage scores.
The following example illustrates that the conditions in Theorem~\ref{t_exact}
are non-trivial.

\begin{example}
In Theorem~\ref{t_exact} it is not always possible to find $k$ columns
from $\bV^T$ that yield an orthogonal matrix.

For instance, let
$$ \bV^T = \begin{pmatrix}  1/2 & 1/2 & 1/2 & 1/2 \\
-1/\sqrt{14} & -2/\sqrt{14} & 3/\sqrt{14} & 0  \end{pmatrix},$$
and $c=\rank(\bV)=2$. Since no
two columns of $\bV^T$ are orthogonal, no two columns can be 
scaled to be orthonormal. Thus no $2\times 2$ matrix 
submatrix of $\bV^T$ can give rise to an orthogonal matrix.

However, for $c=3$ it is possible to construct a $2\times 3$ matrix 
with orthonormal rows. Selecting columns
$t_1=1$, $t_2=2$ and $t_3=3$ from $\bV^T$, and weights
$w_1=\sqrt{5/2}$, $w_2=\sqrt{2/5}$ and $w_3=\sqrt{11/10}$ yields
a matrix
$$ \bV^T\begin{pmatrix}\sqrt{\frac{5}{2}}e_1& \sqrt{\frac{2}{5}}e_2 &
\sqrt{\frac{11}{10}}e_3\end{pmatrix} = \begin{pmatrix}
  \phantom{-}\sqrt{\frac{5}{8}} & \phantom{-}\sqrt{\frac{1}{10}} &
  \phantom{-}\sqrt{\frac{11}{40}} \\ -\sqrt{\frac{5}{28}} &
  -\sqrt{\frac{4}{35}} &
  \phantom{-}\sqrt{\frac{99}{140}} \end{pmatrix}$$ 
that has orthonormal rows.
\end{example}

\begin{remark}[Connection to Theorem~\ref{t_lr}]
In Theorem~\ref{t_exact} the condition $c=k$ implies 
that the $k\times k$ matrix 
$$\bV^T\begin{pmatrix}e_{t_1} & \ldots & e_{t_k}\end{pmatrix}=\bV^T\bS$$
is non-singular.
From Theorem~\ref{t_lr} follows
that $\bW_{opt} =(\bV^T\bS)^{-1}(\bV^T\bS)^{-T}$ is the unique minimal 
Frobenius norm solution for $\bA\bA^T= (\bA\bS)\bW(\bA\bS)^T$.

If, in addition, the rows of $\bV^T\bS\bW_{opt}^{1/2}$ are orthonormal, then 
the minimal norm solution $\bW_{opt}$ is a diagonal matrix,
$$\bW_{opt}=(\bV^T\bS)^{-1}(\bV^T\bS)^{-T}=
\diag\begin{pmatrix} \tfrac{1}{\|\bV^Te_{t_1}\|_2^2} & \cdots &
\tfrac{1}{\|\bV^Te_{t_k}\|_2^2}\end{pmatrix}.$$
\end{remark}

\section{Monte Carlo algorithm for Gram Matrix Approximation}\label{s_randmm}
We review the randomized algorithm to approximate the  Gram matrix (Section~\ref{s_alg});
and discuss and compare two different types of sampling probabilities 
(Section~\ref{s_prob}).

\subsection{The algorithm}\label{s_alg}
The randomized algorithm for approximating $\bA\bA^T$,
presented as Algorithm~\ref{alg_outer}, is a special case of 
the \textsf{BasicMatrixMultiplication} Algorithm
\cite[Figure 2]{DKM2006} which samples according to
the \textsf{Exactly(c)} algorithm \cite[Algorithm 3]{Drineas2010},
that is, independently and with replacement. This means a column can be
sampled more than once. 

A conceptual version of the randomized algorithm is presented as 
Algorithm~\ref{alg_outer}. 
Given a user-specified number of samples $c$, and a set of probabilities $p_j$,
this version assembles columns of the
sampling matrix $\bS$, then applies $\bS$ to $\bA$, and finally 
computes the product 
$$\bX = (\bA\bS)\,(\bA\bS)^T=\sum_{j=1}^c{\frac{1}{cp_{t_j}}\>A_{t_j}A_{t_j}^T}.$$
The choice of weights $1/(c p_{t_j})$ makes $\bX$ 
an unbiased estimator,  $\E[\bX]=\bA\bA^T$ \cite[Lemma 3]{DKM2006}.

\begin{algorithm}
\caption{Conceptual version of randomized matrix multiplication
\cite{DKM2006,Drineas2010}}\label{alg_outer}
\begin{algorithmic}
\REQUIRE $m\times n$ matrix $\bA$, number of samples $1\leq c\leq n$\\
$\qquad $ Probabilities $p_j$, $1\leq j\leq n$, with $p_j\geq 0$ and 
$\sum_{j=1}^n{p_j}=1$
\STATE 
\ENSURE Approximation $\bX=(\bA\bS)\>(\bA\bS)^T$
where $\bS$ is $n\times c$ with $\E[\bS\,\bS^T]=\bI_n$
\STATE
\STATE $\bS=\bzero_{n\times c}$
\FOR{$j=1:c$}
\STATE Sample $t_j$ from $\{1,\ldots,n\}$ with probability $p_{t_j}$
\STATE independently and with replacement
\STATE $S_j= e_{t_j}/\sqrt{c p_{t_j}}$
\ENDFOR
\STATE $\bX=(\bA\bS)\>(\bA\bS)^T$
\end{algorithmic}
\end{algorithm}

Discounting the cost of sampling,
Algorithm \ref{alg_outer} requires $\mathcal{O}(m^2c)$ flops to compute an
approximation to $\bA\bA^T$.
Note that Algorithm~\ref{alg_outer} allows zero probabilities. Since an
index corresponding to $p_j=0$ can never be selected, 
division by zero does not occur in the computation of~$\bS$.
Implementations of sampling with replacement are discussed 
in \cite[Section 2.1]{ESS2011}. For matrices of small 
dimension, one can simply use the Matlab function \texttt{randsample}.

\subsection{Sampling probabilities}\label{s_prob}
We consider two types of probabilities, the ``optimal''
probabilities from \cite{DKM2006} (Section~\ref{s_popt}), 
and leverage score probabilities (Section~\ref{s_plev})
motivated by Corollary~\ref{c_exactg} and Theorem~\ref{t_exact}, and
their use in other randomized algorithms
\cite{BMD09,Drineas2006,Drineas2008}.  We show 
(Theorem~\ref{t_rankone}) that for rank-one
matrices, Algorithm~\ref{alg_outer} with ``optimal'' probabilities
produces the exact result with a single sample.
Numerical experiments (Section~\ref{s_pcomp}) illustrate that sampling 
with ``optimal'' probabilities
results in smaller two-norm relative errors than sampling with leverage
score probabilities, and that the two types of probabilities can differ
significantly.

\subsubsection{``Optimal'' probabilities \cite{DKM2006}}\label{s_popt}
They are defined by
\begin{eqnarray}\label{p_opt}
p_j^{opt} = \frac{\norm{A_j}_2^2}{\norm{\bA}_F^2}, \qquad 1\leq j\leq n
\end{eqnarray}
and are called ``optimal'' because they minimize 
$\E\left[ \norm{\bX-\bA\bA^T}_F^2 \right]$ \cite[Lemma 4]{DKM2006}.  The ``optimal'' probabilities can be computed in  $\mathcal{O}(mn)$ flops.

The analyses in \cite[Section 4.4]{DKM2006} apply to the
more general ``nearly optimal'' probabilities $p_j^{\beta}$, which satisfy
$\sum_{j=1}^n{p^{\beta}_j}=1$ and are constrained by
\begin{eqnarray}\label{p_nopt}
p_j^{\beta} \geq \beta \> p_j^{opt}, \qquad 1\leq j\leq n,
\end{eqnarray}
where $0 < \beta \leq 1$ is a scalar. In the special case $\beta=1$,
they revert to the optimal probabilites,
$p^{\beta}_j=p^{opt}_j$, $1\leq j\leq n$. 
Hence $\beta$ can be viewed as the deviation of the
probabilities $p_j^{\beta}$ from the ``optimal'' probabilities $p^{opt}_j$.

\subsubsection{Leverage score probabilities \cite{BDMI11,BMD09}}\label{s_plev}
The exact representation in Theorem~\ref{t_exact} suggests
probabilities based on the leverage scores of $\bV^T$, 
\begin{eqnarray}\label{p_lev}
p_j^{lev} = \frac{\norm{\bV^Te_j}_2^2}{\norm{\bV}_F^2}
=\frac{\|\bV^Te_j\|_2^2}{k}, \qquad 1\leq j\leq n,
\end{eqnarray}
where $k=\rank(\bA)$.

Since the leverage score probabilities are proportional
to the squared column norms of
$\bV^T$, they are the ``optimal'' probabilities for approximating
$\bV^T\bV = \bI_k$. Exact computation 
of leverage score probabilities, via SVD or QR decomposition,
requires $\mathcal{O}(m^2n)$ flops; 
thus, it is more expensive than the computation of the ``optimal''
probabilities.

In the special case of rank-one matrices, the ``optimal'' and leverage score
probabilities are identical;
and Algorithm~\ref{alg_outer} with ``optimal'' probabilities computes 
the exact result with any number of samples, 
and in particular a single sample. This follows directly
from Corollary~\ref{c_exactg}.

\begin{theorem}\label{t_rankone}
If $\rank(\bA)=1$, then $p^{lev}_j=p^{opt}_j$, $1\leq j\leq n$.

If $\bX$ is computed by Algorithm~\ref{alg_outer} with any $c\geq 1$ and 
probabilities $p^{opt}_j$, then $\bX=\bA\bA^T$.
\end{theorem}

\begin{table}[!ht]
\begin{center}
\begin{tabular}{|c|c|c|c|}
\hline
Dataset & $m\times n$ & $\rank(\bA)$ & $\sr(\bA)$ \\
\hline
Solar Flare  & $10 \times 1389$ & 10 & 1.10 \\
\hline
EEG Eye State  & $15 \times 14980$ & 15 & 1.31 \\
\hline
QSAR biodegradation & $41 \times 1055$ & 41 & 1.13 \\
\hline
Abalone  & $8 \times 4177$ & 8 & 1.002 \\
\hline
Wilt  & $5 \times 4399$  & 5 & 1.03 \\
\hline
Wine Quality - Red & $12 \times 1599$ & 12 & 1.03\\
\hline
Wine Quality - White & $12 \times 4898$ & 12 & 1.01 \\
\hline
Yeast & $8 \times 1484$  & 8 & 1.05\\
\hline
\end{tabular}
\end{center}
\caption{Eight datasets from \cite{BL13}, and the
dimensions, rank and stable rank of the associated matrices~$\bA$.}
\label{tt_datasets}
\end{table}

\subsubsection{Comparison of sampling probabilities}\label{s_pcomp}
We compare the norm-wise relative errors due to randomization 
of Algorithm~\ref{alg_outer} when it samples
with ``optimal'' probabilites and leverage score probabilities.

\paragraph{Experimental set up}
We present experiments with eight representative matrices, described in
Table~\ref{tt_datasets}, from the UCI Machine Learning Repository \cite{BL13}.  

For each matrix, we ran Algorithm~\ref{alg_outer} twice:
once sampling with ``optimal'' probabilities $p_j^{opt}$,
and once sampling with leverage score probabilities $p_j^{lev}$. 
The sampling amounts
$c$ range from 1 to $n$, with 100 runs for each value of $c$.

Figure \ref{f_data1} contains two plots for each matrix:
The left plot shows the two-norm relative errors due to 
randomization, $\|\bX-\bA\bA^T\|_2/\|\bA\bA^T\|_2$, averaged
over 100 runs, versus the sampling amount~$c$.
The right plot shows the ratios of leverage score
over ``optimal'' probabilities $p_j^{lev}/p_j^{opt}$, $1\leq j\leq n$.

\paragraph{Conclusions}
Sampling with  ``optimal'' probabilities produces average
errors that are lower, by as much as a factor of 10, 
than those from sampling with leverage score probabilities,
for all sampling amounts $c$. Furthermore,
corresponding leverage score and ``optimal'' probabilities
tend to differ by several orders of magnitude.

\begin{figure}[!ht]
\begin{center}
\subfigure[Flare]{
\includegraphics[height = 1.0in]{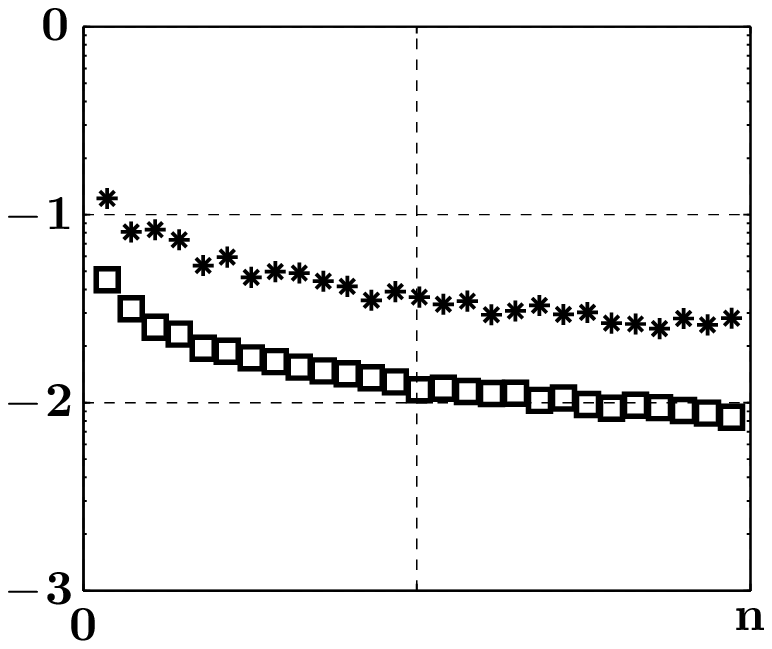}
\includegraphics[height = 1.0in]{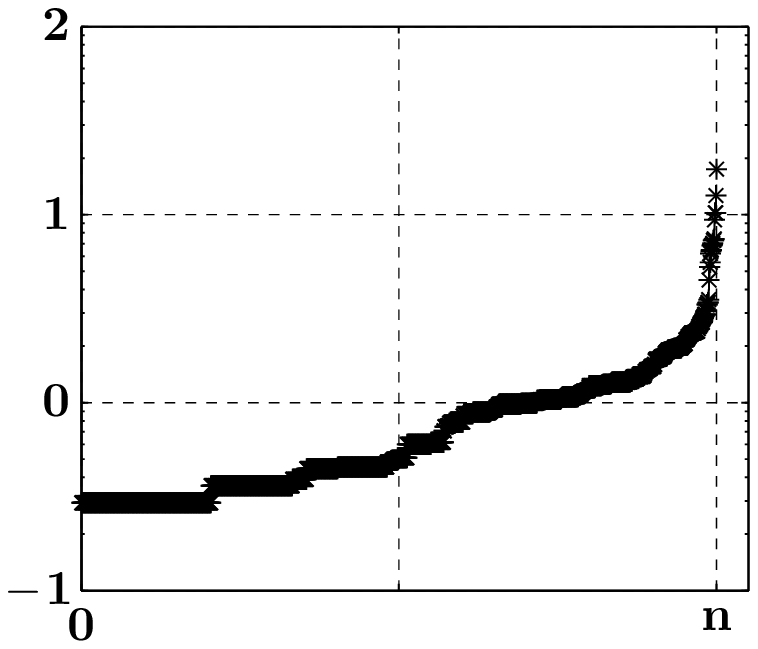}
}
\subfigure[Eye]{
\includegraphics[height = 1.0in]{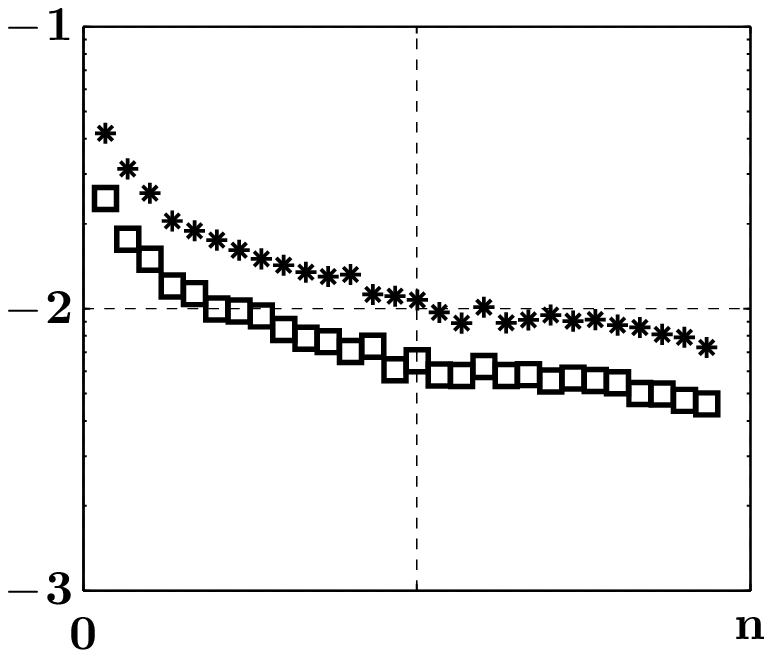}
\includegraphics[height = 1.0in]{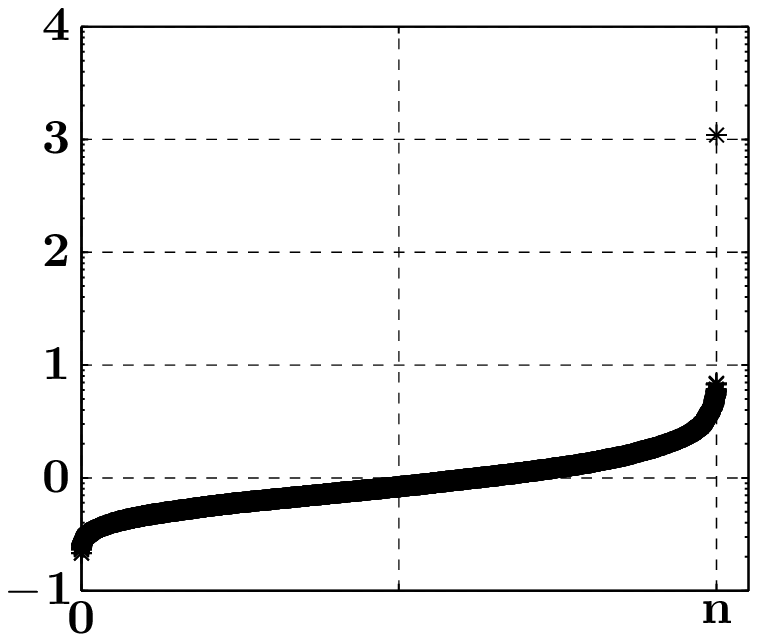}
}
\subfigure[BioDeg]{
\includegraphics[height = 1.0in]{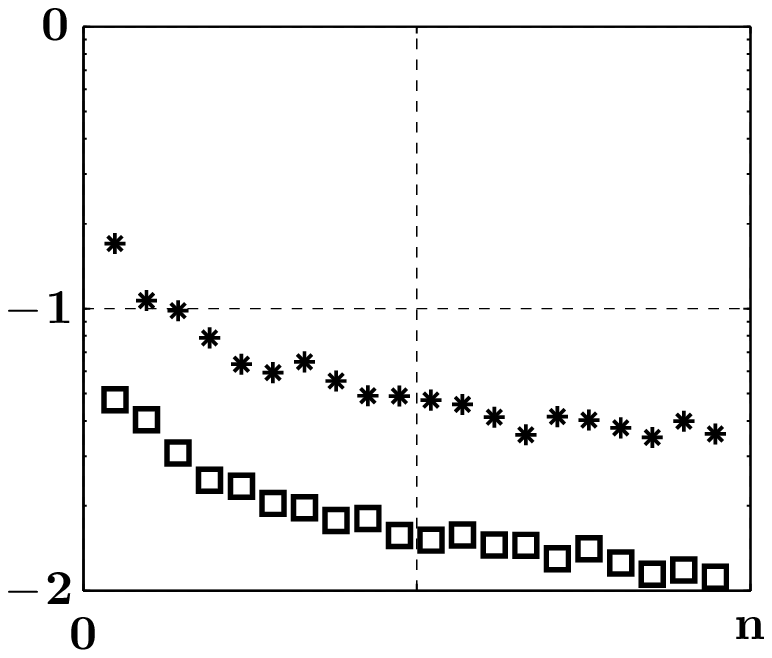}
\includegraphics[height = 1.0in]{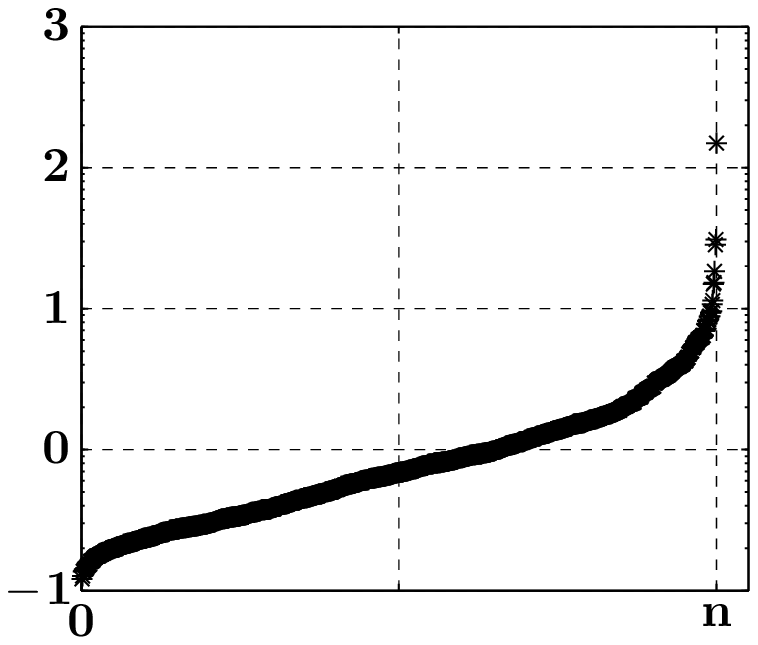}
}
\subfigure[Abalone]{
\includegraphics[height = 1.0in]{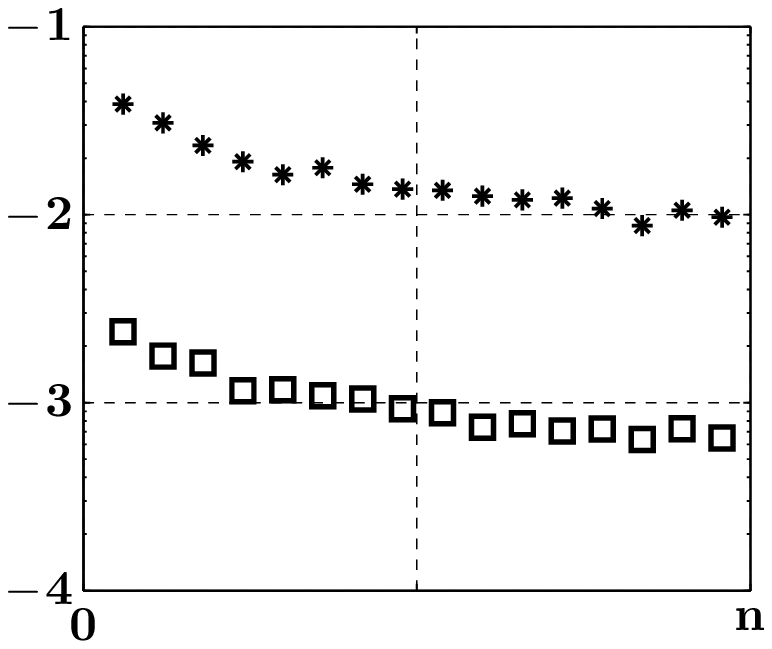}
\includegraphics[height = 1.0in]{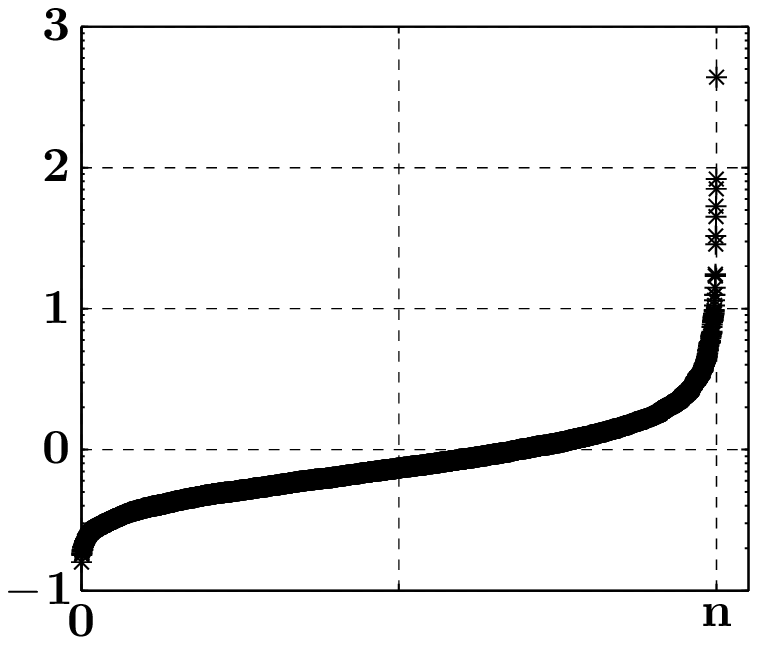}
}
\subfigure[Wilt]{
\includegraphics[height = 1.0in]{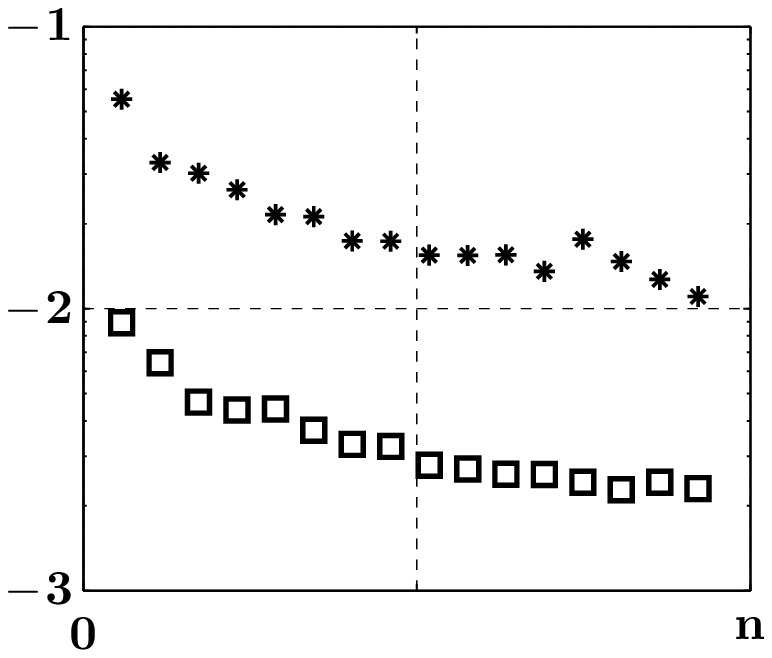}
\includegraphics[height = 1.0in]{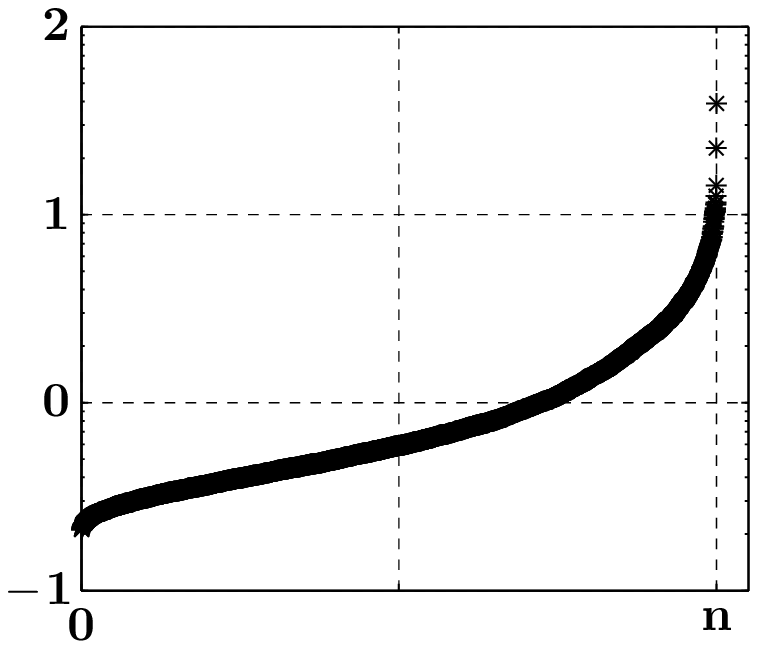}
}
\subfigure[Wine Red]{
\includegraphics[height = 1.0in]{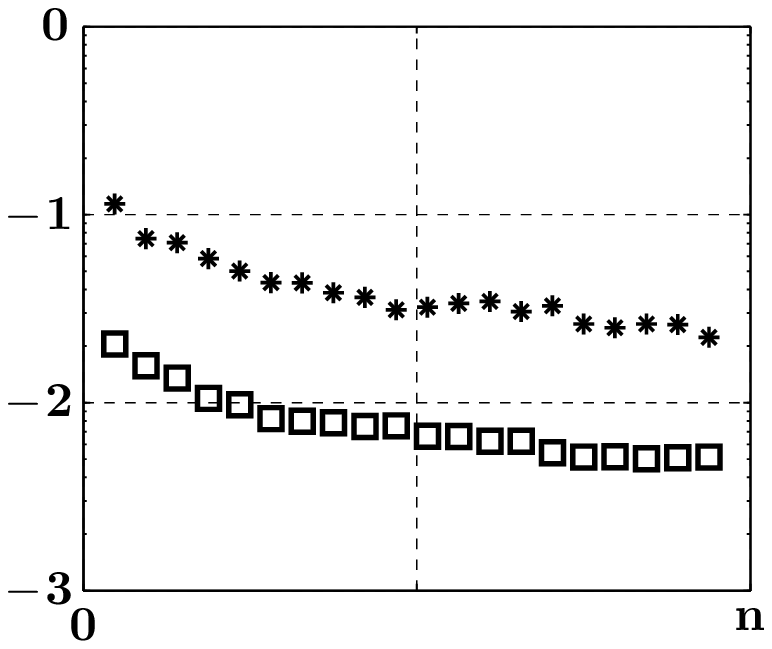}
\includegraphics[height = 1.0in]{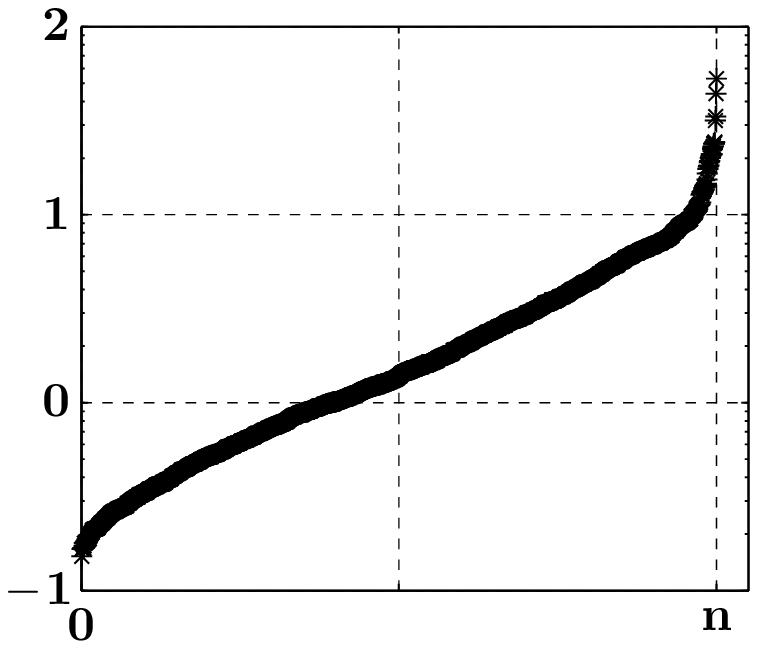}
}
\subfigure[Wine White]{
\includegraphics[height = 1.0in]{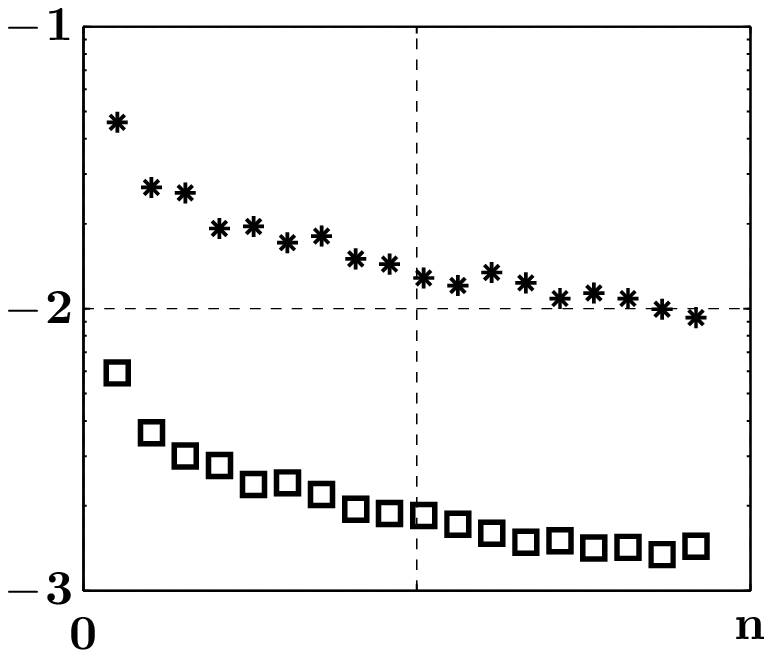}
\includegraphics[height = 1.0in]{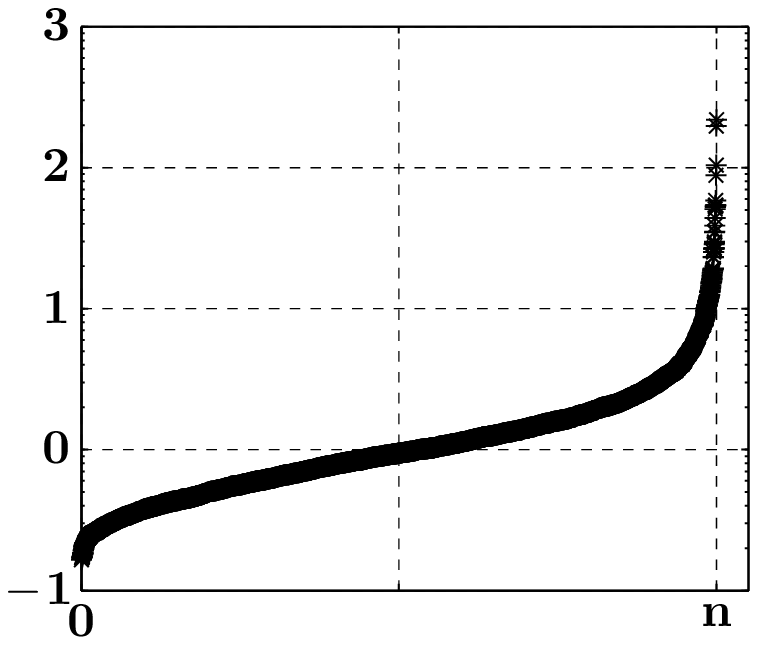}
}
\subfigure[Yeast]{
\includegraphics[height = 1.0in]{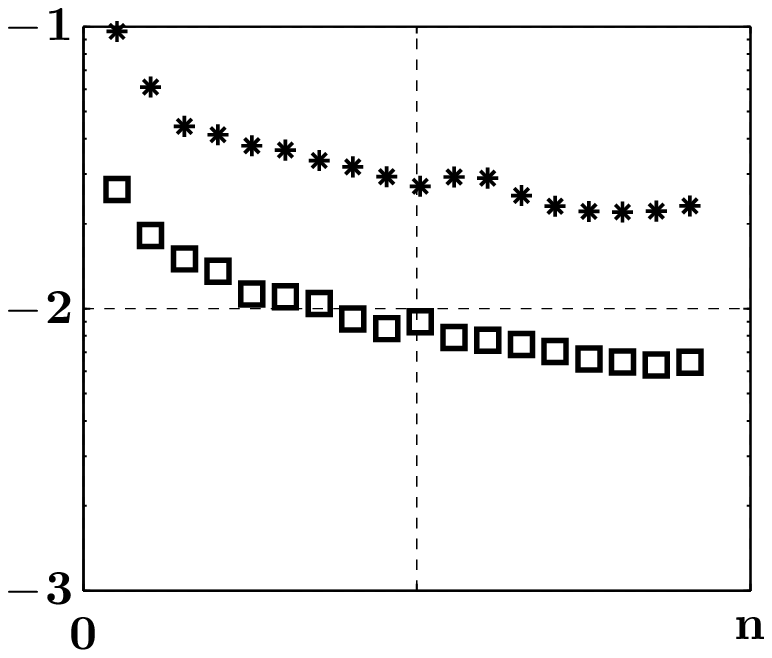}
\includegraphics[height = 1.0in]{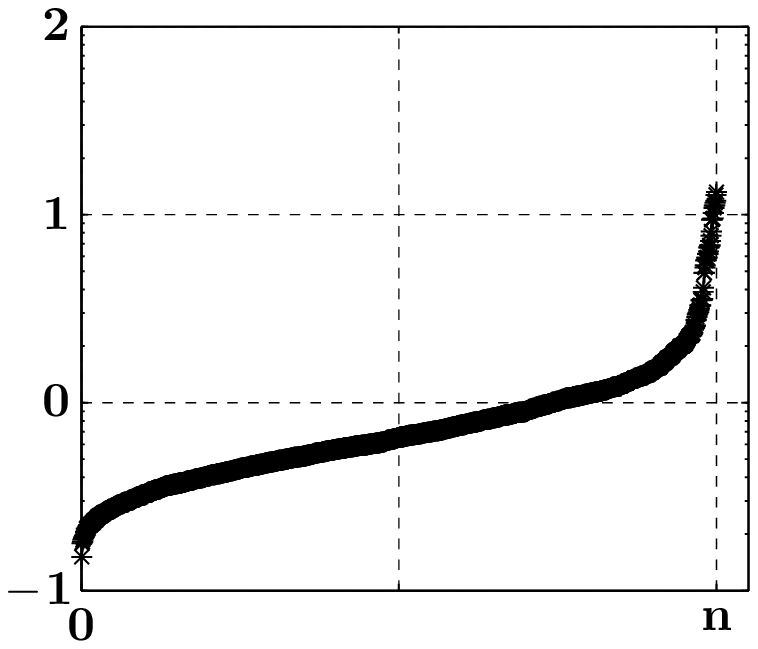}
}
\end{center}

\caption{Relative errors due to randomization, and ratios of leverage score
over ``optimal'' probabilities for the matrices in Table~\ref{tt_datasets}.  
Plots in columns 1 and 3: 
The average over 100 runs of $\norm{\bX-\bA\bA^T}_2/\norm{\bA\bA^T}_2$ 
when Algorithm~\ref{alg_outer} samples with 
``optimal probabilities'' $(\square)$ and with leverage score probabilities
 $(*)$, versus the number $c$ of sampled columns in $\bX$. 
The vertical axes are logarithmic, and the 
labels correspond to powers of 10.  
Plots in columns 2 and 4: Ratios $p_j^{lev}/p_j^{opt}$, $1\leq j\leq n$, 
sorted in increasing magnitude from left to right.}
\label{f_data1}
\end{figure}

\section{Error due to randomization, for sampling with ``nearly optimal''
probabilities}\label{s_error}
We present two new probabilistic bounds (Sections \ref{s_fb} and~\ref{s_sb})
for the two-norm relative error due to randomization,
when Algorithm \ref{alg_outer} samples with the ``nearly optimal''
probabilities in (\ref{p_nopt}). The bounds depend
on the stable rank or the rank of $\bA$, but not on the matrix dimensions.
Neither bound is always better than the other (Section~\ref{s_cbb}).
The numerical experiments (Section~\ref{s_cb})
illustrate that the bounds are informative, even for stringent success
probabilities and matrices of small dimension.

\subsection{First bound}\label{s_fb}
The first bound depends on the stable rank of $\bA$ and also, weakly, on the 
rank.

\begin{theorem}\label{t_troppmult}
Let $\bA\neq \bzero$ be an $m \times n$ matrix, and let $\bX$ be computed
by Algorithm~\ref{alg_outer} with the ``nearly optimal'' probabilities 
$p_j^{\beta}$ in (\ref{p_nopt}).

Given  $0<\delta<1$ and $0<\epsilon \leq 1$, if the number
of columns sampled by Algorithm~\ref{alg_outer} is at least
$$c \geq c_0(\epsilon) \> \sr(\bA) \>
\frac{\ln{(\rank(\bA)/\delta)}}{\beta\,\epsilon^2},
\qquad where \quad
c_0(\epsilon)\equiv 2 + \frac{2\epsilon}{3},$$
then with probability at least $1-\delta$,
$$\frac{\norm{\bX-\bA\bA^T}_2}{\|\bA\bA^T\|_2} \leq \epsilon.$$
\end{theorem}

\begin{proof} See Section~\ref{s_psb}.
\end{proof}

As the required error $\epsilon$ becomes smaller, so does the constant 
$c_0(\epsilon)$ in the lower bound for the number of samples,
that is, $c_0(\epsilon) \rightarrow 2$ as $\epsilon \rightarrow 0$.

\subsection{Second bound}\label{s_sb}
This bound depends only on the stable rank of $\bA$.

\begin{theorem}\label{t_minsker}
Let $\bA\neq \bzero$ be an $m \times n$ matrix, and let $\bX$ be computed
by Algorithm~\ref{alg_outer} with the ``nearly optimal'' probabilities 
$p_j^{\beta}$ in (\ref{p_nopt}).

Given $0 < \delta < 1$ and $0<\epsilon \leq 1$, if the number of 
columns sampled by Algorithm~\ref{alg_outer} is at least
$$c \geq c_0(\epsilon) \> \sr(\bA)\>
\frac{\ln{(4\sr(\bA)/\delta)}}{\beta\,\epsilon^2},
\qquad \mbox{where} \quad
c_0(\epsilon)\equiv 2 + \frac{2\epsilon}{3},$$
then with probability at least $1-\delta$,
$$\frac{\norm{\bX-\bA\bA^T}_2}{\|\bA\bA^T\|_2} \leq \epsilon.$$
\end{theorem}

\begin{proof} See Section~\ref{s_pminsker}.
\end{proof}

\subsection{Comparison}\label{s_cbb}
The bounds in Theorems \ref{t_troppmult} and~\ref{t_minsker} differ
only in the arguments of the logarithms.

On the one hand, Theorem~\ref{t_minsker} is tighter than Theorem
\ref{t_troppmult} if $4\>\sr(\bA) < \rank(\bA)$.
On the other hand, Theorem \ref{t_troppmult} is tighter
for matrices with large stable rank, and in particular for 
matrices $\bA$ with orthonormal rows where $\sr(\bA) = \rank(\bA)$.

In general, Theorem \ref{t_minsker} is tighter than all the bounds in 
Table~\ref{tt_two}, that is, to our knowledge, all published bounds.

\begin{table}[!ht]
\begin{center}
\begin{tabular}{|c|c|c|c|c|c|}
\hline
Matrix & $m\times n$ & $\rank(\bA)$ & $\sr(\bA)$ & $c\,\gamma_1$ & 
$c\,\gamma_2$\\
\hline
\textsf{us04} & $163 \times 28016$ & 115 & 5.27 & 16.43 & 13.44\\
\hline
\textsf{bibd\_16\_8}  & $163 \times 28016$ & 120 & 4.29 & 13.43 & 10.65\\
\hline
\end{tabular}\medskip
\end{center}
\caption{Matrices from \cite{DavisHu11}, their dimensions, rank and stable 
rank;  and key quantities from (\ref{e_thh1}) and~(\ref{e_thh2}).}
\label{tt_comp}
\end{table}

\subsection{Numerical experiments}\label{s_cb}
We compare the bounds in Theorems \ref{t_troppmult} and~\ref{t_minsker}
to the errors of Algorithm~\ref{alg_outer} for sampling with 
``optimal'' probabilities.

\paragraph{Experimental set up}
We present experiments  with two matrices from the University
of Florida Sparse Matrix Collection \cite{DavisHu11}. The matrices have
the same dimension, and 
similar high ranks and low stable ranks, see Table~\ref{tt_comp}.
Note that only for low stable ranks can Algorithm~\ref{alg_outer} achieve 
any accuracy.

The sampling amounts $c$ range from 1 to $n$, the number of columns, 
with 100 runs for each value of $c$. From the 100 
errors $\|\bX-\bA\bA^T\|_2/\|\bA\bA^T\|_2$ for each $c$ value, we plot 
the smallest, largest, and average.

In Theorems \ref{t_troppmult} and~\ref{t_minsker},
the success probability is 99 percent, 
that is, a failure probability of $\delta=.01$. The error 
bounds are plotted as a function of $c$. That is, for
Theorem~\ref{t_troppmult} we plot (see Theorem~\ref{t_tropprank})
\begin{eqnarray}\label{e_thh1}
\frac{\norm{\bX-\bA\bA^T}_2}{\|\bA\bA^T\|_2} \leq
\gamma_1+\sqrt{\gamma_1 \>(6 +\gamma_1)}, \qquad 
\gamma_1\equiv\sr(\bA)\>\frac{\ln{(\rank(\bA)/.01)}}{3\,\,c}
\end{eqnarray}
while for Theorem~\ref{t_minsker} we plot (see Theorem~\ref{t_minskermult})
\begin{eqnarray}\label{e_thh2}
\frac{\norm{\bX-\bA\bA^T}_2}{\|\bA\bA^T\|_2} \leq
\gamma_2+\sqrt{\gamma_2 \>(6 +\gamma_2)}, \qquad 
\gamma_2\equiv\>\sr(\bA)\>\frac{\ln{(4\sr(\bA)/.01)}}{3\,c}
\end{eqnarray}
The key quantities $c\>\gamma_1$ and $c\>\gamma_2$ are shown for both matrices
in Table~\ref{tt_comp}.

Figure \ref{f_comparebounds} contains two plots,
the left one for matrix \textsf{us04}, and the right one for 
matrix \textsf{bibd\_16\_8}. The plots show the relative errors 
$\|\bX-\bA\bA^T\|_2/\|\bA\bA^T\|_2$ and the bounds (\ref{e_thh1})
and (\ref{e_thh2}) versus the sampling amount $c$.

\paragraph{Conclusions}
In both plots, the bounds corresponding
to Theorems \ref{t_troppmult} and~\ref{t_minsker}
are virtually indistinguishable, as was is already predicted
by the key quantities $c\>\gamma_1$ and $c\>\gamma_2$ in Table~\ref{tt_comp}.
The bounds overestimate the worst case error from Algorithm~\ref{alg_outer} 
by a factor of at most 10.
Hence they are informative, even for matrices of small dimension
and a stringent success probability.

\begin{figure}
\begin{center}
\includegraphics[height = 2in]{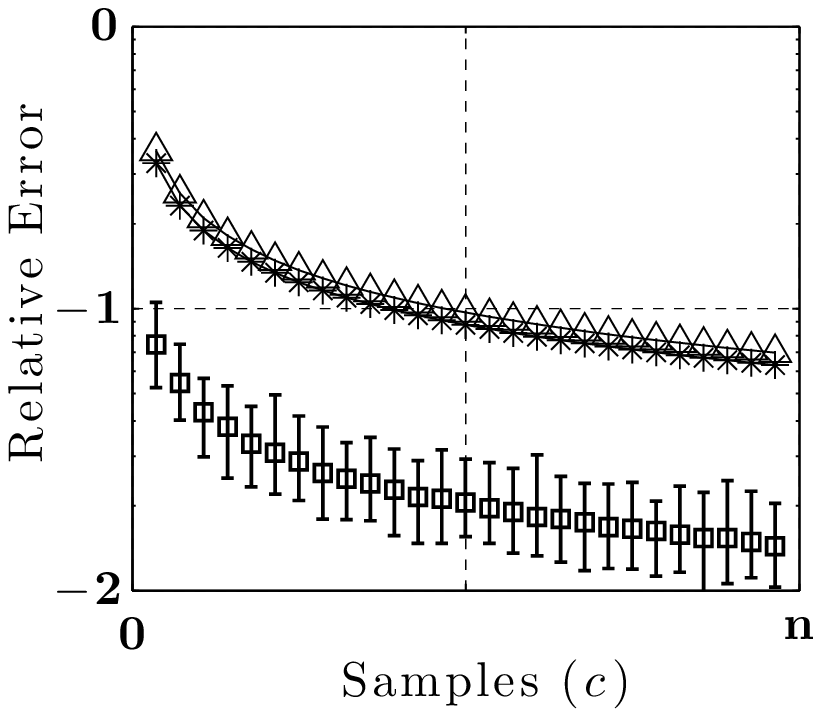} 
\includegraphics[height = 2in]{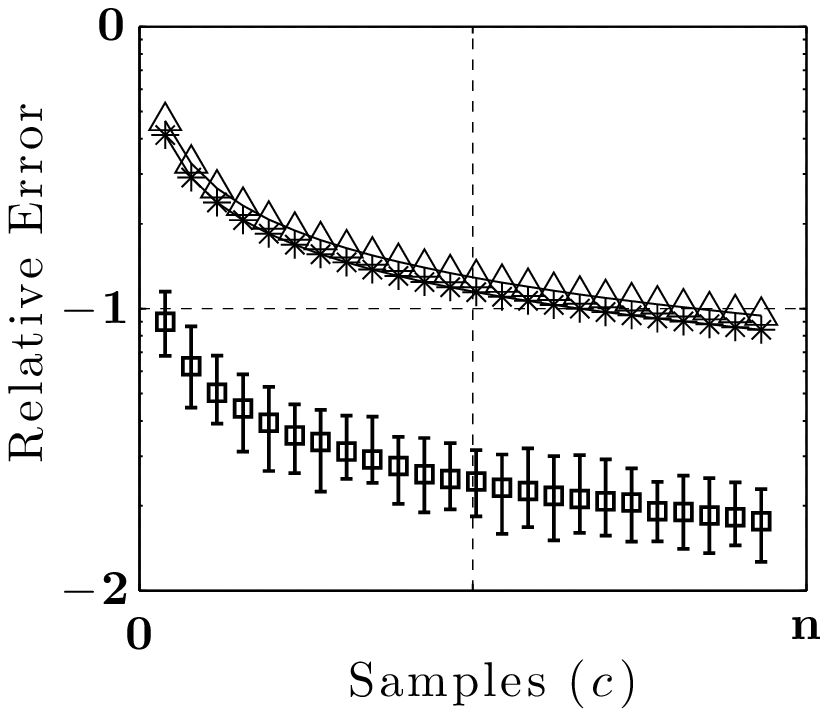}
\end{center}
\caption{Relative errors due to randomization from Algorithm~\ref{alg_outer},
and bounds (\ref{e_thh1}) and (\ref{e_thh2}) versus sampling amount $c$,
for matrices \textsf{us04} (left)  
and \textsf{bidb\_16\_8} (right).  
Error bars represent the maximum and minimum of the errors
$\norm{\bX-\bA\bA^T}_2/\norm{\bA\bA^T}_2$ 
from Algorithm~\ref{alg_outer}
over 100 runs, while the squares represent the average.  
The triangles $(\bigtriangleup)$ represent the bound (\ref{e_thh1}), while
the stars $(*)$ represent (\ref{e_thh2}).
The vertical axes are logarithmic, and the labels correspond to powers of 10.}
\label{f_comparebounds}
\end{figure}
\section{Error due to randomization, for sampling with leverage score
probabilities}\label{s_error2}
For completeness, 
we present a normwise relative bound for the error due to randomization,
when Algorithm \ref{alg_outer} samples with leverage score
probabilities (\ref{p_lev}). The bound 
corroborates the numerical experiments in Section~\ref{s_pcomp},
and suggests that sampling with leverage score probabilities produces 
a larger error due to randomization than sampling with 
``nearly optimal'' probabilities.

\begin{theorem}\label{t_levremove}
Let $\bA\neq \bzero$ be an $m \times n$ matrix, and let $\bX$ be computed
by Algorithm~\ref{alg_outer} with the leverage score probabilites 
$p_j^{lev}$ in (\ref{p_lev}).

Given $0<\delta<1$ and $0<\epsilon\leq 1$,  if the number of columns
sampled by Algorithm~\ref{alg_outer} is at least
$$ c \geq c_0(\epsilon)\>\rank(\bA)\>
\frac{\ln(\rank(\bA)/\delta)}{\epsilon^2}, 
\qquad \mbox{where} \quad c_0(\epsilon) =  2 + \frac{2\epsilon}{3},$$
then with probability at least $1-\delta$,
$$ \frac{\norm{\bX-\bA\bA^T}_2}{\norm{\bA\bA^T}_2} \leq \epsilon.$$
\end{theorem}

\begin{proof} See Section~\ref{s_plevremove}.
\end{proof}

In the special case when $\bA$  has orthonormal columns, 
the leverage score probabilities $p_j^{lev}$ are equal to the ``optimal''
probabilities $p_j^{opt}$ in (\ref{p_opt}). Furthermore,
$\rank(\bA)=\sr(\bA)$, so that Theorem~\ref{t_levremove} is equal to 
Theorem~\ref{t_troppmult}.
For general matrices $\bA$, though, $\rank(\bA)\geq \sr(\bA)$, 
and Theorem~\ref{t_levremove} is not as tight as Theorem~\ref{t_troppmult}.
\section{Singular value and condition number bounds}\label{s_sing}
As in \cite{Drineas2010}, we apply the bounds for the Gram
matrix approximation to a matrix with orthonormal rows, and 
derive bounds for the smallest singular value (Section~\ref{s_sv})
and condition number (Section~\ref{s_cond}) of a sampled matrix.

Specifically, let $\bQ$ be a real $m \times n$ matrix with
orthonormal rows, $\bQ\bQ^T = \bI_m$. Then,
as discussed in Section~\ref{s_popt}, the ``optimal'' probabilities
(\ref{p_opt}) for $\bQ$ are equal to the leverage score 
probabilities (\ref{p_lev}),
$$p_j^{opt}=\frac{\|Q_j\|_2^2}{\|\bQ\|_F^2}=
\frac{\|Q_j\|_2^2}{m}=p_j^{lev},\qquad 1\leq j\leq m.$$

The connection between Gram matrix approximations 
$(\bQ\bS)\>(\bQ\bS)^T$ and singular values 
of the sampled matrix $\bQ\bS$ comes from the well-conditioning of
singular values \cite[Corollary 2.4.4]{Golub2013},
\begin{eqnarray}\label{e_connect}
\left| 1 - \sigma_j\left( \bQ\bS \right)^2 \right| &= &
\left| \sigma_j\left(\bQ\bQ^T\right) - 
\sigma_j\left( (\bQ\bS)\>(\bQ\bS)^T \right) \right|\notag \\
&\leq& \left\|\bQ\bQ^T - (\bQ\bS)\>(\bQ\bS)^T\right\|_2,
\qquad 1\leq j\leq m.
\end{eqnarray}

\subsection{Singular value bounds}\label{s_sv}
We present two bounds for the smallest singular value of a sampled matrix,
for sampling with the ``nearly optimal'' probabilities (\ref{p_nopt}),
and for uniform sampling with and without replacement.

The first bound is based on the Gram matrix approximation in
Theorem~\ref{t_troppmult}.

\begin{theorem}\label{t_singmatmult}
Let $\bQ$ be an $m \times n$ matrix with orthonormal rows and
coherence~$\mu$, and let $\bQ\bS$ be computed by Algorithm~\ref{alg_outer}.
Given $0<\epsilon < 1$ and $0<\delta<1$, we have 
$\sigma_m\left(\bQ\bS \right) \geq \sqrt{1-\epsilon}$
with probability at least $1-\delta$, if Algorithm~\ref{alg_outer} 
\begin{itemize}
\item either samples with the ``nearly optimal'' probabilities $p_j^{\beta}$, and
$$c \geq c_0(\epsilon)\>  m\>\frac{\ln(m/\delta)}{\beta \epsilon^2},$$
\item or samples with uniform probabilities $1/n$, and
$$c \geq c_0(\epsilon)\> n\,\mu\>\frac{\ln(m/\delta)}{\epsilon^2}.$$
\end{itemize}
Here  $c_0(\epsilon)\equiv 2 + \tfrac{2}{3}\>\epsilon$.
\end{theorem}

\begin{proof}
See Section~\ref{s_psingmatmult}.
\end{proof}

Since $c_0(\epsilon) \geq 2$, the above bound for uniform sampling 
is slightly less tight than the last bound
in Table~\ref{tt_sing}, i.e. \cite[Lemma 1]{Git11}.
Although that bound technically holds only for uniform
sampling \textit{without} replacement, the same proof
gives the same bound for uniform sampling \textit{with} replacement.

This inspired us to 
derive a second bound, by modifying the argument in \cite[Lemma 1]{Git11},
to obtain a slightly tighter constant.
This is done with a direct application of a Chernoff bound
(Theorem~\ref{t_tcher}).
The only difference between the next and the previous result
is the smaller constant~$c_1(\epsilon)$, and the added application to sampling 
without replacement.

\begin{theorem}\label{t_singchernoff}
Let $\bQ$ be an $m \times n$ matrix with orthonormal rows and
coherence~$\mu$, and let $\bQ\bS$ be computed by Algorithm~\ref{alg_outer}.
Given $0<\epsilon < 1$ and $0<\delta<1$, we have 
$\sigma_m\left(\bQ\bS \right) \geq \sqrt{1-\epsilon}$
with probability at least $1-\delta$, if Algorithm~\ref{alg_outer} 
\begin{itemize}
\item either samples with the ``nearly optimal'' probabilities $p_j^{\beta}$, and
$$c \geq c_1(\epsilon)\>  m\>\frac{\ln(m/\delta)}{\beta \epsilon^2},$$
\item or samples with uniform probabilities $1/n$, with or without
replacement, and
$$c \geq c_1(\epsilon)\> n\,\mu\>\frac{\ln(m/\delta)}{\epsilon^2}.$$
\end{itemize}
Here  $c_1(\epsilon)\equiv 
\tfrac{\epsilon^2}{(1-\epsilon)\ln(1-\epsilon)+\epsilon}$, and
$1\leq c_1(\epsilon)\leq 2$.
\end{theorem}
\smallskip

\begin{proof}
See Section~\ref{s_psingchernoff}.
\end{proof}

The constant $c_1(\epsilon)$ is slightly smaller than the constant~2 
in~\cite[Lemma 1]{Git11}, which is the last bound in Table~\ref{tt_sing}.

\subsection{Condition number bounds}\label{s_cond}
We present two bounds for the condition number  
$\kappa(\bQ\bS) \equiv \sigma_1(\bQ\bS)/\sigma_m(\bQ\bS)$ of a sampled
matrix $\bQ\bS$ with full row-rank.

The first condition number bound is based on a Gram matrix approximation,
and is analogous to Theorem~\ref{t_singmatmult}.

\begin{theorem}\label{t_condmatmult}
Let $\bQ$ be an $m \times n$ matrix with orthonormal rows and
coherence~$\mu$, and let $\bQ\bS$ be computed by Algorithm~\ref{alg_outer}.
Given $0<\epsilon < 1$ and $0<\delta<1$, we have 
$\kappa(\bQ\bS) \leq \tfrac{\sqrt{1+\epsilon}}{\sqrt{1-\epsilon}}$
with probability at least $1-\delta$, if Algorithm~\ref{alg_outer} 
\begin{itemize}
\item either samples with the ``nearly optimal'' probabilities $p_j^{\beta}$, and
$$c \geq c_0(\epsilon)\>  m\>\frac{\ln(m/\delta)}{\beta \epsilon^2},$$
\item or samples with uniform probabilities $1/n$, and
$$c \geq c_0(\epsilon)\> n\,\mu\>\frac{\ln(m/\delta)}{\epsilon^2}.$$
\end{itemize}
Here  $c_0(\epsilon)\equiv 2 + \tfrac{2}{3}\>\epsilon$.
\end{theorem}

\begin{proof}
See Section~\ref{s_pcondmatmult}.
\end{proof}

The second condition number bound is based on a Chernoff inequality,
and is analogous to Theorem~\ref{t_singchernoff}, but with a different 
constant, and an additional factor of two in the logarithm.

\begin{theorem}\label{t_condchernoff}
Let $\bQ$ be an $m \times n$ matrix with orthonormal rows and
coherence~$\mu$, and let $\bQ\bS$ be computed by Algorithm~\ref{alg_outer}.
Given $0<\epsilon < 1$ and $0<\delta<1$, we have 
$\kappa(\bQ\bS) \leq \tfrac{\sqrt{1+\epsilon}}{\sqrt{1-\epsilon}}$
with probability at least $1-\delta$, if Algorithm~\ref{alg_outer} 
\begin{itemize}
\item either samples with the ``nearly optimal'' probabilities $p_j^{\beta}$, and
$$c \geq c_2(\epsilon)\>  m\>\frac{\ln(2m/\delta)}{\beta \epsilon^2},$$
\item or samples with uniform probabilities $1/n$, with or without
replacement, and
$$c \geq c_2(\epsilon)\> n\,\mu\>\frac{\ln(2m/\delta)}{\epsilon^2}.$$
\end{itemize}
Here  $c_2(\epsilon)\equiv 
\tfrac{\epsilon^2}{(1+\epsilon)\ln(1+\epsilon)-\epsilon}$, and
$2\leq c_2(\epsilon)\leq 2.6$.
\end{theorem}
\smallskip

\begin{proof}
See Section~\ref{s_pcondchernoff}.
\end{proof}

It is difficult to compare the two condition number bounds,
and neither bound is always tighter than the other.  
On the one hand, Theorem~\ref{t_condchernoff} has a smaller constant than 
Theorem~\ref{t_condmatmult} since $c_2(\epsilon)\leq c_1(\epsilon)$. 
On the other hand, though, Theorem \ref{t_condmatmult} has an
additional factor of two in the logarithm.
For very large $m/\delta$, the additional factor of 2 in the
logarithm does not matter much and Theorem~\ref{t_condchernoff} is
tighter.  

In general, Theorem~\ref{t_condchernoff} is not always tighter
than Theorem~\ref{t_condmatmult}. For example, 
if $m=100$, $\delta = 0.01$, $\epsilon = 0.1$, $\beta= 1$, and 
Algorithm~\ref{alg_outer} 
samples with ``nearly optimal'' probabilities, then
Theorem~\ref{t_condchernoff} 
requires $1.57 \cdot 10^{5}$ samples, while
Theorem~\ref{t_condmatmult} requires only $1.43 \cdot 10^{5}$; 
hence, it is tighter.

\section*{Acknowledgements}
We thank Petros Drineas and Michael Mahoney for useful discussions, and the four anonymous reviewers whose suggestions helped us to improve the quality of the 
paper.
\section{Proofs}\label{s_proofs}
We present proofs for the results in Sections  \ref{s_det} --~\ref{s_sing}.

\subsection{Proof of Theorem~\ref{t_lr}}\label{s_plr}
We will use the two lemmas below.
The first one is a special case of \cite[Theorem 2.1]{FrT07} where the
rank of the approximation is not restricted.
 
\begin{lemma}\label{l_fr}
Let $\bH$ be $m\times n$, $\bB$ be $m\times p$,
and $\bC$ be $q\times n$ matrices, 
and let $\bP_{\bB}$ be the orthogonal projector onto $\range(\bB)$, and 
$\bP_{\bC^T}$ the orthogonal projector onto $\range(\bC^T)$. Then the solution of
$$\min_{\bW}{\|\bH-\bB\,\bW\,\bC\|_F}$$
with minimal Frobenius norm is 
$$\bW=\bB^{\dagger}\>\bP_{\bB}\,\bH\,\bP_{\bC^T}\>\bC^{\dagger}.$$
\end{lemma}

\begin{lemma}\label{l_mp}
If $\bB$ is $m\times p$ and $\bC$ is $p\times n$, with 
$\rank(\bB)=p=\rank(\bC)$, then $(\bB\bC)^{\dagger}=\bC^{\dagger}\bB^{\dagger}$.
\end{lemma}

\begin{proof}
Set $\bY\equiv \bB\bC$, and use $\bB^{\dagger}\bB=\bI_p=\bC\bC^{\dagger}$ to 
verify
that $\bZ\equiv \bC^{\dagger}\bB^{\dagger}$ satisfies
the four conditions defining the Moore-Penrose inverse
\begin{eqnarray}\label{e_mp}
\bY\bZ\bY=\bY, \quad \bZ\bY\bZ=\bZ, \quad (\bY\bZ)^T=\bY\bZ, \quad 
(\bZ\bY)^T=\bZ\bY.
\end{eqnarray}
\end{proof}

\subsubsection*{Proof of Theorem~\ref{t_lr}}
Abbreviate $\bA_1\equiv\bA\bS$ and $\bV_1^T\equiv \bV^T\bS$.

In Lemma~\ref{l_fr}, set $\bH=\bA\bA^T$, $\bB=\bA_1$, and $\bC=\bA_1^T$. Then
$\bP_{\bB}=\bA_1\bA_1^{\dagger}=\bP_{\bC^T}$, and
$$\bW_{opt}=\bA_1^{\dagger}\>\bA_1\bA_1^{\dagger}\,\bA\bA^T\,
\bA_1\bA_1^{\dagger}\>(\bA_1^{\dagger})^T.$$
The conditions for the Moore-Penrose inverse (\ref{e_mp}) imply
$\bA_1^{\dagger}\bA_1\bA_1^{\dagger}= \bA_1^{\dagger}$, and 
$$\bA_1\bA_1^{\dagger}\> (\bA_1^{\dagger})^T=
\left(\bA_1\bA_1^{\dagger}\right)^T\>(\bA_1^{\dagger})^T=
(\bA_1^{\dagger})^T \>\bA_1^T\>(\bA_1^{\dagger})^T=(\bA_1^{\dagger})^T.$$
Hence  $\bW_{opt}=\bA_1^{\dagger}\>\bA\bA^T\>(\bA_1^{\dagger})^T$.

\paragraph{Special case $\rank(\bA_1)=\rank(\bA)$} This means 
the number of columns $c$ in $\bA_1=\bU\bSigma\bV_1^T$ is at least as large as 
$k\equiv \rank(\bA)$.
Hence $\bV_1^T$ is $k\times c$ with $c\geq k$, and 
$\rank(\bV_1^T)=k=\rank(\bU\bSigma)$.
From Lemma~\ref{l_mp} follows
$\bA_1^{\dagger}=(\bV_1^{\dagger})^T\>\bSigma^{-1}\bU^T$. Hence
$$\bW_{opt}=(\bV_1^{\dagger})^T\>\bV^T\bV\>\bV_1^{\dagger} =
(\bV_1^{\dagger})^T\>\bV_1^{\dagger}.$$
Furthermore $\rank(\bA_1)=\rank(\bA)$ implies that $\bA_1$ has the same
column space as $\bA$. Hence the residual 
in Theorem~\ref{t_lr} is zero, and $\bA_1\bW_{opt}\bA_1^T=\bA\bA^T$.

\paragraph{Special case $c=\rank(\bA_1)=\rank(\bA)$} This means 
$c=k$, so that $\bV_1$ is a $k\times k$ matrix. From $\rank(\bA)=k$
follows $\rank(\bV_1)=k$, so that $\bV_1$ is nonsingular
and $\bV_1^{\dagger}=\bV_1^{-1}$.

\subsection{Proof of Theorem~\ref{t_exactg}}\label{s_pexactg}
Abbreviate
$$\bA_1\equiv \begin{pmatrix}A_{t_1}&\cdots &A_{t_c}\end{pmatrix}, \qquad
\bV_1^T\equiv
\bV^T\begin{pmatrix}e_{t_1} & \cdots & e_{t_c}\end{pmatrix},$$
so that the sum of outer products can be written as
$\sum_{j=1}^c{w_j\>A_{t_j}A_{t_j}^T}= \bA_1\bW\bA_1^T$, 
where $\bW\equiv \diag\begin{pmatrix}w_1 & \cdots & w_c\end{pmatrix}$.

\paragraph{1. Show: If $\bA_1\bW\bA_1^T=\bA\bA^T$ for a diagonal $\bW$ with non-negative diagonal,
then $\bV_1^T\bW^{1/2}$ has orthonormal rows}
From $\bA\bA^T=\bA_1\bW\bA_1^T$ follows
\begin{eqnarray}\label{e_pexactg}
\bU\bSigma^2\bU^T=\bA\bA^T=\bA_1\bW\bA_1^T=
\bU\bSigma \> \bV_1^T\,\bW\,\bV_1\>\bSigma \bU^T.
\end{eqnarray}
Multiplying by $\bSigma^{-1}\bU^T$ on the left and by 
$\bU\bSigma^{-1}$ on the right gives $\bI_k=\bV_1^T\,\bW\, \bV_1$.
Since $\bW$ is positive semi-definite,
it has a symmetric positive semi-definite square root $\bW^{1/2}$. Hence
$\bI_k=\bV_1^T\,\bW\, \bV_1=(\bV_1^T\bW^{1/2})\>(\bV_1^T\bW^{1/2})^T$,
and $\bV_1^T\bW^{1/2}$ has orthonormal rows.

\paragraph{2. Show: If $\bV_1^T\bW^{1/2}$ has orthonormal rows, then
$\bA_1\bW\bA_1^T=\bA\bA^T$}
Inserting
$\bI_k=(\bV_1^T\bW^{1/2})\>(\bV_1^T\bW^{1/2})^T=\bV_1^T\bW\bV_1$
into $\bA_1 \bW\bA_1^T$ gives
\begin{eqnarray*}
\bA_1 \bW\bA_1^T
=\bU\bSigma \>\left(\bV_1^T \,\bW\,\bV_1\right)\>\bSigma \bU^T=
\bU\bSigma^2 \bU^T = \bA\bA^T.
\end{eqnarray*}

\subsection{Proof of Corollary~\ref{c_exactg}}\label{s_pcexactg}
Since $\rank(\bA)=1$,
the right singular vector matrix
$\bV=\begin{pmatrix}v_1 & \ldots & v_n\end{pmatrix}^T$ is a $n\times 1$
vector.
Since $\bA$ has only a single non-zero singular value,
$\|A_j\|_2=\|\bU\bSigma\>v_j\|_2=\|\bA\|_F v_j$.
Clearly $A_j\neq 0$ if and only $v_j\neq 0$, and 
$\|\bV^Te_j\|_2^2=v_j^2=\|A_j\|_2^2/\|\bA\|_F^2$.
Let $A_{t_j}$ be any $c$ non-zero columns of $\bA$. Then
$$\sum_{j=1}^c{w_jA_{t_j}A_{t_j}^T}=\bU\bSigma \>
\left(\sum_{j=1}^c{w_j v_{t_j}^2}\right)\>\bSigma\bU^T=\bU\bSigma^2\bU^T=
\bA\bA^T$$
if and only if 
$\sum_{j=1}^c{w_j v_{t_j}^2}=1$. This is true if 
$w_j=1/(cv_{t_j}^2)$, $1\leq j\leq c$.

\subsection{Proof of Theorem~\ref{t_exact}}\label{s_pexact}
Since Theorem~\ref{t_exact} is a special case of Theorem \ref{t_exactg},
we only need to derive the expression for the weights.
From $c=k$ follows that $\bV_1^T\bW^{1/2}$ is $k\times k$
with orthonormal rows. Hence $\bV_1^T\bW^{1/2}$ is an orthogonal matrix,
and must have orthonormal columns as well,
$(\bW^{1/2}\bV_1)\> (\bW^{1/2}\bV_1)^T = \bI_k$. Thus
$$ \bV_1\bV_1^T = \diag
\begin{pmatrix}\norm{\bV^Te_{t_1}}_2^2 &\cdots & \norm{\bV^Te_{t_c}}_2^2 
\end{pmatrix} = \bW^{-1}.$$
This and $\bW^{1/2}$ being diagonal implies 
$w_j = 1/\norm{\bV^Te_{t_j}}_2^2$.

\subsection{Proof of Theorem~\ref{t_troppmult}}\label{s_psb}
We present two auxiliary results, a matrix Bernstein concentration 
inequality (Theorem~\ref{t_tropp1}) and a bound for the singular values 
of a difference of positive semi-definite matrices
(Theorem~\ref{t_zhan}), before deriving
a probabilistic bound (Theorem~\ref{t_tropp}). The subsequent combination of
Theorem~\ref{t_tropp} and the invariance of the two-norm under unitary
transformations yields Theorem~\ref{t_tropprank} which, at last, leads 
to a proof for the desired Theorem~\ref{t_troppmult}.

\begin{theorem}[Theorem 1.4 in \cite{Tropp2011b}]\label{t_tropp1}
Let $\bX_j$ be $c$ independent real symmetric random $m\times m$ matrices.  Assume that, with probability one, $\E[\bX_j]=\bzero$, $1\leq j\leq c$ and 
$\max_{1\leq j\leq c}{\|\bX_j\|_2}\leq \rho_1$. Let
$\left\|\sum_{j=1}^c{\E[\bX_j^2]}\right\|_2\leq \rho_2$.

Then for any $\epsilon\geq0$
$$\Prob\left[\left\|\sum_{j=1}^c{\bX_j}\right\|_2\geq \epsilon\right]\leq
m\>\exp\left(-\frac{\epsilon^2/2}{\rho_2+\rho_1\epsilon/3}\right).$$
\end{theorem}

\begin{theorem}[Theorem 2.1 in \cite{Zhan2000}]\label{t_zhan}
If $\bB$ and $\bC$ are $m\times m$ real symmetric positive semi-definite
matrices, with singular values 
$\sigma_1(\bB)\geq \ldots \geq \sigma_m(\bB)$ 
and $\sigma_1(\bC)\geq \ldots \geq \sigma_m(\bC)$, 
then the singular values of the difference are bounded by
$$\sigma_j(\bB-\bC)\leq 
\sigma_j\begin{pmatrix}\bB & \bzero \\ \bzero & \bC\end{pmatrix},
\qquad 1\leq j\leq m.$$
In particular, $\|\bB-\bC\|_2\leq \max\{\|\bB\|_2,\,\|\bC\|_2\}$.
\end{theorem}

\begin{theorem}\label{t_tropp}
Let $\bA\neq \bzero$ be an $m \times n$ matrix, and let $\bX$ be computed
by Algorithm~\ref{alg_outer} with the ``nearly optimal'' probabilites 
$p_j^{\beta}$ in (\ref{p_nopt}).

For any $\delta>0$, with probability at least $1-\delta$,
$$\frac{\norm{\bX-\bA\bA^T}_2}{\|\bA\bA^T\|_2} \leq
\gamma_0+\sqrt{\gamma_0 \>(6 +\gamma_0)}, \qquad where \quad
\gamma_0\equiv\sr(\bA)\>\frac{\ln{(m/\delta)}}{3\,\beta\,c}.$$
\end{theorem}

\begin{proof} In order to apply Theorem~\ref{t_tropp1}, 
we need to change variables, and check that the assumptions are satisfied.

\paragraph{1. Change of variables}
Define the $m\times m$ real symmetric matrix random variables
$\bY_j\equiv \frac{1}{c\,p_{t_j}}\> A_{t_j}A_{t_j}^T$, and
write the output of Algorithm~\ref{alg_outer} as 
$$\bX=(\bA\bS)\>(\bA\bS)^T=\bY_1+\cdots +\bY_c.$$
Since $\E[\bY_j]=\bA\bA^T/c$, but
Theorem~\ref{t_tropp1} requires random variables with zero mean, set
$\bX_j \equiv \bY_j -\frac{1}{c}\bA\bA^T$.
Then
$$\bX-\bA\bA^T=(\bA\bS)\,(\bA\bS)^T-\bA\bA^T=
\sum_{j=1}^c{\left(\bY_j-\frac{1}{c}\,\bA\bA^T\right)}=\sum_{j=1}^c{\bX_j}.$$
Hence, we show 
$\left\|\bX-\bA\bA^T\right\|_2 \leq \epsilon$ by showing 
$\left\|\sum_{j=1}^c{\bX_j}\right\|_2\leq \epsilon$.

Next we have to check that the assumptions of Theorem~\ref{t_tropp1}
are satisfied.
In order to derive bounds for $\max_{1\leq j\leq c}{\|\bX_j\|_2}$ and 
$\norm{\sum_{j=1}^c{\E[\bX_j^2]}}_2$, we 
assume general non-zero probabilities $p_j$ for the moment, that is,
$p_j>0$, $1\leq j\leq n$.
\medskip

\paragraph{2. Bound for $\max_{1\leq j\leq c}{\|\bX_j\|_2}$}
Since $\bX_j$ is a difference of positive semidefinite matrices,
apply Theorem~\ref{t_zhan} to obtain
$$\|\bX_j\|_2 \leq \max 
\left\{\|\bY_j\|_2,\,\tfrac{1}{c}\norm{\bA\bA^T}_2\right\}\leq 
\frac{\hat{\rho}_1}{c},
\qquad  \hat{\rho}_1\equiv
\max_{1\leq i\leq n}\left\{\frac{\|A_i\|_2^2}{p_i},\,\norm{\bA}_2^2\right\}.$$
\medskip

\paragraph{3. Bound for $\norm{\sum_{j=1}^c{\E[\bX_j^2]}}_2$}
To determine the expected value of
$$\bX_j^2 = \bY_j^2-
\tfrac{1}{c}\,\bA\bA^T\>\bY_j-\tfrac{1}{c}\bY_j\>\bA\bA^T + 
\tfrac{1}{c^2}(\bA\bA^T)^2$$
use the linearity of the expected value and $\E[\bY_j]=\bA\bA^T/c$ to obtain
$$\E[\bX_j^2] =\E[\bY_j^2] -\frac{1}{c^2}\,(\bA\bA^T)^2.$$
Applying the definition of expected value again yields
$$\E[\bY_j^2]=\frac{1}{c^2}\>\sum_{i=1}^n{p_i\>\frac{(A_iA_i^T)^2}{p^2_i}}
= \frac{1}{c^2}\>\sum_{i=1}^n{\frac{(A_iA_i^T)^2}{p_i}}.$$ 
Hence
\begin{eqnarray*}
\sum_{j=1}^c{\E[\bX_j^2]} &=& \frac{1}{c}\>
\left(\sum_{i=1}^n{\frac{(A_iA_i^T)^2}{p_i}} -(\bA\bA^T)^2\right)=
\frac{1}{c} \bA\>
\left(\sum_{i=1}^n{e_i \frac{\|A_i\|_2^2}{p_i}e_i^T}-\bA^T\bA\right)\>\bA^T\\
&=& \frac{1}{c} \bA\>(\bL-\bA^T\bA)\>\bA^T,
\end{eqnarray*}
where 
$\bL\equiv \diag\begin{pmatrix}\|A_1\|_2^2/p_1 & \ldots & \|A_n\|_2^2/p_n
\end{pmatrix}$.
Taking norms and applying Theorem~\ref{t_zhan} to $\|\bL-\bA^T\bA\|_2$ gives
$$\norm{\sum_{j=1}^c{\E[\bX_j^2]}}_2 \leq
\frac{\|\bA\|_2^2}{c}\>\max\left\{\|\bL\|_2,\,\|\bA\|_2^2\right\}
=\frac{\|\bA\|_2^2}{c}\>\hat{\rho}_1.$$

\paragraph{4. Application of Theorem~\ref{t_tropp1}}
The required upper bounds for Theorem~\ref{t_tropp1} are
$$\|\bX_j\|_2 \leq \rho_1\equiv\frac{\hat{\rho}_1}{c} \qquad and \qquad
\norm{\sum_{j=1}^c{\E[\bX_j^2]}}_2 \leq 
\rho_2\equiv \frac{\|\bA\|_2^2}{c}\>\hat{\rho}_1.$$
Inserting these bounds into Theorem~\ref{t_tropp1} gives
$$\Prob\left[\left\| \sum_{j=1}^c{\bX_j} \right\|_2> \epsilon\right]\leq 
m\> \exp\left( \frac{-c\epsilon^2}{2\hat{\rho}_1 \,
(\|\bA\|_2^2 +\epsilon/3)}\right).$$
Hence $\left\| \sum_{j=1}^c{\bX_j} \right\|_2\leq \epsilon$
with probability at least $1-\delta$, where
$$\delta\equiv m\> \exp\left( \frac{-c\epsilon^2}{2\hat{\rho}_1 \,
(\|\bA\|_2^2 +\epsilon/3)}\right).$$
Solving for $\epsilon$ gives
$$\epsilon=\tau_1\, \hat{\rho}_1+\sqrt{\tau_1\, \hat{\rho}_1 \>
\left(6 \|\bA\|_2^2+\tau_1\, \hat{\rho}_1\right)}, \qquad
\tau_1\equiv \frac{\ln{(m/\delta)}}{3c}.$$

\paragraph{5. Specialization to ``nearly optimal'' probabilities}
We remove zero columns from the matrix. This does not change the norm
or the stable rank. The
``nearly optimal'' probabilities for the resulting submatrix are
$p_j^{\beta}=\beta\|A_j\|_2^2/\|\bA\|_F^2$, with $p_j>0$ for all $j$.
Now 
replace $p_j^{\beta}$  by their lower bounds (\ref{p_nopt}). This gives
$\hat{\rho}_1\leq \|\bA\|_2^2\,\tau_2$ 
where $\tau_2\equiv\sr(\bA)/\beta\geq 1$, and
$$\epsilon\leq \|\bA\|_2^2
\left(\tau_1 \tau_2+\sqrt{\tau_1 \tau_2 \>
\left(6 +\tau_1\tau_2\right)}\right).$$
Finally observe that $\gamma_0=\tau_1\tau_2$, and divide by 
$\|\bA\|_2^2=\|\bA\bA^T\|_2$.
\end{proof}

We make Theorem~\ref{t_tropp} tighter and replace the dimension~$m$
by $\rank(\bA)$. The idea is to apply Theorem~\ref{t_tropp}
to the $k\times k$ matrix $(\bSigma \bV^T)\>(\bSigma \bV^T)^T$ instead
of the $m\times m$ matrix $\bA\bA^T$.

\begin{theorem}\label{t_tropprank}
Let $\bA\neq \bzero$ be an $m \times n$ matrix, and let $\bX$ be computed
by Algorithm~\ref{alg_outer} with the ``nearly optimal'' probabilites 
$p_j^{\beta}$ in (\ref{p_nopt}).

For any $\delta>0$, with probability at least $1-\delta$,
$$\frac{\norm{\bX-\bA\bA^T}_2}{\|\bA\bA^T\|_2} \leq
\gamma_1+\sqrt{\gamma_1 \>(6 +\gamma_1)}, \qquad where \quad
\gamma_1\equiv\sr(\bA)\>\frac{\ln{(\rank(\bA)/\delta)}}{3\,\beta\,c}.$$
\end{theorem}

\begin{proof}
The invariance of the two-norm under unitary transformations implies
$$\|\bX -\bA\bA^T\|_2=
\norm{(\bSigma \bV^T\bS)\>(\bSigma \bV^T\bS)^T-(\bSigma \bV^T)\>
(\bSigma \bV^T)^T}_2.$$
Apply Theorem~\ref{t_tropp} to the $k\times n$ matrix $B\equiv\bSigma \bV^T$
with probabilities 
$$p_j^{\beta}\geq \beta\>\frac{\|A_j\|_2^2}{\|\bA\|_F^2}=
\beta\>\frac{\|B_j\|_2^2}{\|\bB\|_F^2}.$$
\end{proof}

Note that Algorithm~\ref{alg_outer} is still applied to the original
matrix $\bA$, with probabilities (\ref{p_nopt}) computed from
$\bA$. It is only the bound that has changed.  

\subsubsection*{Proof of Theorem~\ref{t_troppmult}}
At last, we set $\gamma_1+\sqrt{\gamma_1 \>(6 +\gamma_1)} \leq \epsilon$ and 
solve for~$c$ as follows.
In $\gamma_1+\sqrt{\gamma_1 \>(6 +\gamma_1)}$, write
$$\gamma_1=\tfrac{\ln{(\rank(\bA)/\delta)}}{3\,\beta\,c}\>\sr(\bA)=
\tfrac{t}{3c}, \qquad \mbox{where} \quad
t\equiv \frac{\ln{(\rank(\bA)/\delta)}\>\sr(\bA)}{\beta}.$$
We want to determine $\alpha>0$ so that $c=\alpha t/\epsilon^2$ satisfies
$$\gamma_1 +\sqrt{\gamma_1\>(6+\gamma_1)}=
\frac{t}{3c}+\sqrt{\frac{t}{3c}\left(6+\frac{t}{3c}\right)}\leq \epsilon.$$
Solving for $\alpha$ gives $\alpha \geq 2+2\epsilon/3=c_0(\epsilon)$.

\subsection{Proof of Theorem~\ref{t_minsker}}\label{s_pminsker}
To start with, we need a matrix Bernstein concentration inequality,
along with the the L\"{o}wner  partial ordering \cite[Section 7.7]{HJ2013}.  
and the instrinsic dimension \cite[Section 7]{Tropp2012}.

If $\bA_1$ and $\bA_2$ are $m \times m$ real symmetric matrices,
then $\bA_1 \preceq \bA_2$ means
that $\bA_2-\bA_1$ is positive semi-definite
\cite[Definition 7.7.1]{HJ2013}. The \textit{intrinsic dimension} 
of a $m\times m$ symmetric positive semi-definite matrix $\bA$
is  \cite[Definition 7.1.1]{Tropp2012}:
$$\intdim(\bA) \equiv \trace(\bA)/\norm{\bA}_2,$$
where $1\leq \intdim(\bA)\leq \rank(\bA)\leq m$.

\begin{theorem}[Theorem 7.3.1 and (7.3.2) in \cite{Tropp2012}]\label{t_minskerconc}
Let $\bX_j$ be $c$ independent real symmetric random  matrices, with
$\E[\bX_j]=\bzero$, $1\leq j\leq c$. Let 
$\max_{1\leq j\leq c}{\|\bX_j\|_2}\leq \rho_1$, and 
let $\bP$ be a symmetric positive semi-definite matrix so that 
$\sum_{j=1}^c{\E[\bX_j^2]} \preceq  \bP$.
Then for any $\epsilon \geq \norm{\bP}_2^{1/2} + \rho_1/3$
$$\Prob\left[\left\|\sum_{j=1}^c{\bX_j}\right\|_2\geq \epsilon \right] \leq 
4 \>\intdim(\bP)\>\exp\left( \frac{-\epsilon^2/2}{\norm{\bP}_2 + 
\rho_1\epsilon/3} \right).$$
\end{theorem}

Now we apply the above theorem to sampling with ``nearly optimal'' 
probabilities.

\begin{theorem}\label{t_minskermult}
Let $\bA\neq \bzero$ be an $m \times n$ matrix, and let $\bX$ be computed
by Algorithm~\ref{alg_outer} with the ``nearly optimal'' probabilities 
$p_j^{\beta}$ in (\ref{p_nopt}).

For any $0 < \delta < 1$, with probability at least $1-\delta$,
$$\frac{\norm{\bX-\bA\bA^T}_2}{\|\bA\bA^T\|_2} \leq
\gamma_2+\sqrt{\gamma_2 \>(6 +\gamma_2)}, \qquad where \quad
\gamma_2\equiv\>\sr(\bA)\>\frac{\ln{(4\sr(\bA)/\delta)}}{3\,\beta\,c}.$$
\end{theorem}

\begin{proof}
In order to apply Theorem~\ref{t_minskerconc},
we need to change variables, and check that the assumptions are satisfied.

\paragraph{1. Change of variables}
As in item 1 of the proof of Theorem~\ref{t_tropp}, we
define the real symmetric matrix random variables
$\bY_j\equiv \frac{1}{c\,p_{t_j}}\> A_{t_j}A_{t_j}^T$, and
write the output of Algorithm~\ref{alg_outer} as 
$$\bX=(\bA\bS)\>(\bA\bS)^T=\bY_1+\cdots +\bY_c.$$
The zero mean versions are $\bX_j \equiv \bY_j -\frac{1}{c}\bA\bA^T$,
so that
$\bX-\bA\bA^T=\sum_{j=1}^c{\bX_j}$. 

Next we have to check that the assumptions of Theorem~\ref{t_minskerconc} are
satisfied, for the ``nearly optimal'' probabilities 
$p_j^{\beta}= \beta \|A_j\|_2^2/\|\bA\|_F^2$. Since Theorem~\ref{t_minskerconc}
does not depend on the matrix dimensions, we can assume that all
zero columns of $\bA$ have been removed, so that all $p_j^{\beta}>0$.
\medskip

\paragraph{2. Bound for $\max_{1\leq j\leq c}{\|\bX_j\|_2}$}
From item~2 in the proof of Theorem~\ref{t_tropp} follows
$\|\bX_j\|_2 \leq \rho_1$, where
$$\rho_1= \frac{1}{c}
\max_{1\leq j\leq n}\left\{\frac{\|A_j\|_2^2}{p_j^{\beta}},\,
\norm{\bA}_2^2\right\} \leq \frac{\|\bA\|_F^2}{\beta c}.$$
\medskip

\paragraph{3. The matrix~$\bP$}
From item~3 in the proof of Theorem~\ref{t_tropp} follows
$$\sum_{j=1}^c{\E[\bX_j^2]} = \tfrac{1}{c}\bA\bL\bA^T - 
\tfrac{1}{c}\bA\bA^T\bA\bA^T,$$
where 
$\bL\equiv \diag\begin{pmatrix}\|A_1\|_2^2/p_1^{\beta} & \cdots & 
\|A_n\|_2^2/p_n^{\beta}\end{pmatrix}\preceq (\norm{\bA}_F^2/\beta)\>\bI_n$.
Since $\bA\bA^T\bA\bA^T$ is positive semi-definite, so is
$$\tfrac{1}{c}\bA\bA^T\bA\bA^T
=\tfrac{1}{c}\bA\bL\bA^T - \tfrac{1}{c}\left(\bA\bL\bA^T - 
\bA\bA^T\bA\bA^T\right) =\tfrac{1}{c}\bA\bL\bA^T - 
\sum_{j=1}^c{\E[\bX_j^2]}.$$
Thus, 
$\sum_{j=1}^c{\E[\bX_j^2]}  \preceq \frac{1}{c}\bA\bL\bA^T \preceq 
\frac{\norm{\bA}_F^2}{\beta c}\bA\bA^T$,
where the the second inequality follows from \cite[Theorem 7.7.2(a)]{HJ2013}. 
Set $\bP\equiv \frac{\norm{\bA}_F^2}{\beta c}\bA\bA^T$. Then
$$\norm{\bP}_2 = \frac{\|\bA\|_2^2\|\bA\|_F^2}{\beta c} \qquad and \qquad
\intdim(\bP) = 
\frac{\norm{\bA}_F^4}{\norm{\bA}_F^2\norm{\bA}_2^2} = \sr(\bA).$$

\paragraph{4. Application of Theorem~\ref{t_minskerconc}} 
Substituting the above expressions for $\|\bP\|_2$, $\intdim(\bP)$ and
$\rho_1= \frac{\|\bA\|_F^2}{\beta \, c}$
into Theorem~\ref{t_minskerconc} gives
$$\Prob\left[\left\|\sum_{j=1}^c{\bX_j}\right\|_2\geq \epsilon\right]
\leq 4\> \sr(\bA)\>\exp\left( \frac{-\epsilon^2\beta
  c}{2\norm{\bA}_F^2\left(\norm{\bA}_2^2 + \epsilon/3\right)}\right).$$ 
Hence $\left\|\sum_{j=1}^c{\bX_j}\right\|_2\leq\epsilon$ with probability
at least $1-\delta$, where 
$$\delta\equiv 4\> \sr(\bA)\>\exp\left( \frac{-\epsilon^2\beta
  c}{2\norm{\bA}_F^2\left(\norm{\bA}_2^2 + \epsilon/3\right)}\right).$$ 
Solving for $\epsilon$ gives
$$\epsilon =  \hat{\gamma}_2 + 
\sqrt{\hat{\gamma}_2\> (6\norm{\bA}_2^2+\hat{\gamma}_2)}, 
\qquad \mbox{where} \quad 
\hat{\gamma}_2 \equiv \norm{\bA}_F^2\>\frac{\ln(4\,\sr(\bA)/\delta)}{3 \beta c}
=\|\bA\|_2^2\>\gamma_2.$$
It remains to show the last requirement of Theorem~\ref{t_minskerconc},
that is, $\epsilon\geq \norm{\bP}^{1/2}_2 +\rho_1/3$.
Replacing $\epsilon$ by its above expression in terms of $\hat{\gamma}_2$
shows that the requirement is true if 
$\hat{\gamma}_2\geq \rho_1/3$ and 
$\sqrt{6 \|\bA\|_2^2\,\hat{\gamma_2}}\geq \|\bP\|_2^{1/2}$. This is the
case if $\ln(4\,\sr(\bA)/\delta)>1$. 
Since $\sr(\bA)\geq 1$, this is definitely true if $\delta<4/e$.
Since we assumed $\delta<1$ from the start, the requirement is fulfilled
automatically.

At last, divide both sides of
$\norm{\bX - \bA\bA^T}_2 \leq \hat{\gamma}_2 + 
\sqrt{\hat{\gamma}_2\> (6\norm{\bA}_2^2+\hat{\gamma}_2)}$ 
by $\norm{\bA\bA^T}_2 = \norm{\bA}_2^2$.
\end{proof}

\subsubsection*{Proof of Theorem~\ref{t_minsker}}
As in the proof of Theorem \ref{t_troppmult}, solve for $c$ in
$\gamma_2 + \sqrt{\gamma_2\>(6+\gamma_2)} \leq \epsilon$.

\subsection{Proof of Theorem~\ref{t_levremove}}\label{s_plevremove}
To get a relative error bound, substitute the thin SVD
$\bA=\bU\bSigma\bV^T$ into
\begin{eqnarray*}
\|\bX-\bA \bA^T\|_2 &=&\|(\bA\bS)\, (\bA\bS)^T-\bA\bA^T\|_2=
\|(\bSigma\bV^T\bS)\,(\bSigma\bV^T\bS)^T-\bSigma\bV^T\bV\bSigma\|_2\\
&\leq & \|\bSigma\|_2^2 \>\|(\bV^T\bS)\,(\bV^T\bS)^T-\bV^T\bV\|_2\\
&=&\|\bA\bA^T\|_2 \>\|(\bV^T\bS)\,(\bV^T\bS)^T-\bV^T\bV\|_2.
\end{eqnarray*}
The last term can be viewed as sampling columns from $\bV^T$ 
to approximate the product $\bV^T\bV=\bI_n$.
Now apply Theorem~\ref{t_troppmult}, where $\|\bV\|_F^2=k=\rank(\bA)$ and
$\|\bV\|_2^2=1$, so that $\sr(\bV)=k=\rank(\bA)$.

\subsection{Proof of Theorem \ref{t_singmatmult}}\label{s_psingmatmult}
We present separate proofs for the two types of sampling probabilities.

\paragraph{Sampling with ``nearly optimal'' probabilities}
Applying Theorem \ref{t_troppmult} shows that 
$\norm{\bQ\bQ^T - (\bQ\bS)\>(\bQ\bS)^T}_2 \leq \epsilon$
with probability at least $1-\delta$, if
$c \geq c_0(\epsilon)\>\tfrac{m}{\beta \epsilon^2}\>\ln(m/\delta)$.

\paragraph{Sampling with uniform probabilities}
Use the $\beta$ factor to express the uniform probabilities as
``nearly optimal'' probabilities,
$$\frac{1}{n}=\frac{m}{n\>\mu}\> \frac{\mu}{m}
\geq \frac{m}{n\> \mu}\>\frac{\norm{Q_j}_2^2}{\norm{\bQ}_F^2}=
\beta\> \frac{\norm{Q_j}_2^2}{\norm{\bQ}_F^2} = \beta \>p_j^{opt}
\qquad 1\leq j\leq n.$$
Now apply Theorem \ref{t_troppmult} with $\beta = m/(n\mu)$.

For both sampling methods, the connection (\ref{e_connect}) implies that 
$\sigma_m(\bQ\bS) \geq \sqrt{1-\epsilon}$ with probability at least $1-\delta$.

\subsection{Proof of Theorem \ref{t_singchernoff}}\label{s_psingchernoff}
First we present the concentration inequality on which the proof is 
based. Below $\lambda_{min}(\bX)$ and $\lambda_{max}(\bX)$
denote the smallest and largest eigenvalues, respectively,
of the symmetric positive semi-definite matrix $\bX$.

\begin{theorem}[Theorem 5.1.1 in \cite{Tropp2012}]\label{t_tcher}
Let $\bX_j$ be $c$ independent $m \times m$ real symmetric positive
semi-definite random matrices, with 
$\max_{1\leq j\leq c}{\|\bX_j\|_2}\leq \rho$. Define
$$\rho_{max}\equiv \lambda_{max}\left( \E\left[ \sum_{j=1}^c{X_j} \right]\right),
\qquad 
\rho_{min}\equiv \lambda_{min}\left( \E\left[ \sum_{j=1}^c{X_j} \right]\right),$$
and $f(x)\equiv e^{x}/(1+x)^{1+x}$.
Then, for any $0 < \epsilon < 1$
$$\Prob\left[ \lambda_{min}\left(\sum_{j=1}^c{X_j} \right)
\leq (1-\epsilon)\rho_{min}\right] \leq m\>f(-\epsilon)^{\rho_{min}/\rho},$$
and 
$$\Prob\left[\lambda_{max}\left(\sum_{j=1}^c{X_j} \right)
\geq (1+\epsilon)\rho_{max}\right] \leq m\>f(\epsilon)^{\rho_{max}/\rho}.$$
\end{theorem}

\subsubsection*{Proof of Theorem \ref{t_singchernoff}}
Write $(\bQ\bS)\>(\bQ\bS)^T=\sum_{j=1}^c{X_j}$, where
$\bX_j \equiv \tfrac{Q_{t_j}Q_{t_j}^T}{c \> p_{t_j}}$.  
To apply Theorem~\ref{t_tcher} we need to compute $\rho$, $\rho_{min}$,
and $\rho_{max}$.

\paragraph{Sampling with ``nearly optimal'' probabilities}
The definition of ``nearly optimal'' probabilities (\ref{p_nopt})
and the fact that $\norm{\bQ}_F^2 = m$ imply
$\norm{\bX_j}_2 = \tfrac{\norm{Q_{t_j}}_2^2}{cp_{t_j}^\beta} \leq 
\tfrac{m}{c\>\beta}$. Hence we can set $\rho\equiv\tfrac{m}{c\>\beta}$. 
The definition of $\bX_j$ implies
$$ \E\left[ \sum_{j=1}^c{X_{t_j}} \right]  = 
\frac{1}{c}\sum_{j=1}^c{\sum_{i=1}^n{Q_iQ_i^T}} = \bQ\bQ^T = \bI_m,$$
so that $\rho_{min} = 1$.  
Now apply Theorem~\ref{t_tcher} to conclude
\begin{equation*}
\Prob\left[\lambda_{min}\left(\sum_{j=1}^c{X_j} \right) 
\leq (1-\epsilon)\right] \leq m f(-\epsilon)^{c\beta /m}.
\end{equation*}
Setting the right hand side equal to $\delta$ and solving for $c$ gives
$$ c = \frac{m}{\beta}\>\frac{\ln(\delta/m)}{\ln{f(-\epsilon)}}=
c_1(\epsilon)\>m\>\frac{\ln(m/\delta)}{\beta\epsilon^2},$$
where the second equality follows from $\ln{f(x)}= x-(1+x)\ln{(1+x)}$.
The function $c_1(x)$ is decreasing in $[0,1]$, 
and L'H\^{o}pital's rule implies that 
$c_1(\epsilon) \rightarrow 2$ as $\epsilon \rightarrow 0$ and $c_1(\epsilon) \rightarrow 1$ as $\epsilon \rightarrow 1$.

\paragraph{Sampling with uniform probabilities}
An analogous proof with $p_j=1/n$ shows that
$\|\bX_j\|_2\leq \rho\equiv n\mu/c$.

\paragraph{Uniform sampling without replacement}
Theorem~\ref{t_tcher} also holds when the matrices $\bX_j$ are sampled
uniformly without replacement \cite[Theorem 2.2]{Tropp2011}.

For all  three sampling methods,
the connection (\ref{e_connect}) implies that 
$\sigma_m(\bQ\bS) \geq \sqrt{1-\epsilon}$ with probability at least $1-\delta$.

\subsection{Proof of Theorem \ref{t_condmatmult}}\label{s_pcondmatmult}
The proof follows from Theorem~\ref{t_singmatmult}, and the
connection (\ref{e_connect}), since
$|1-\sigma_j^2(\bQ\bS) | \leq \epsilon$, $1\leq j\leq m$, implies that both,
$\sigma_m(\bQ\bS) \geq \sqrt{1-\epsilon}$ and 
$\sigma_1\left( \bQ\bS\right) \leq \sqrt{1+\epsilon}$.

\subsection{Proof of Theorem \ref{t_condchernoff}}\label{s_pcondchernoff}
We derive separate bounds for the smallest and largest 
singular values of $\bQ\bS$.  

\paragraph{Sampling with ``nearly optimal'' probabilities}
The proof Theorem~\ref{t_singchernoff} implies that
\begin{equation*}
\Prob\left[\lambda_{min}\left(\sum_{j=1}^c{X_j} \right) 
\leq (1-\epsilon)\right] \leq m f(-\epsilon)^{c\beta /m}.
\end{equation*}
Similarly, 
we can apply Theorem~\ref{t_tcher} with $\rho_{max}=1$ to conclude
\begin{equation*}
\Prob\left[\lambda_{max}\left(\sum_{j=1}^c{X_j} \right) 
\geq (1+\epsilon)\right] \leq m f(\epsilon)^{c\beta /m}.
\end{equation*}
Since $f(-\epsilon)\leq f(\epsilon)$, Boole's inequality implies
\begin{equation*}
\Prob\left[\lambda_{min}\left(\sum_{j=1}^c{X_j} \right) 
\leq (1-\epsilon)\ \mbox{and}\ 
\lambda_{max}\left(\sum_{j=1}^c{X_j} \right) 
\geq (1+\epsilon)\right] \leq 2m f(\epsilon)^{c\beta /m}.
\end{equation*}
Hence, $\sigma_{m}(\bQ\bS) \geq \sqrt{1-\epsilon}$ and
$\sigma_1(\bQ\bS) \leq \sqrt{1+\epsilon}$ hold simultaneously
with probability at least $1-\delta$, if 
$$c \geq c_2(\epsilon) \> m \>\frac{\ln(2m/\delta)}{\beta \epsilon^2}.$$ 
This bound for $c$ also ensures that
$\kappa(\bQ\bS) \leq \frac{\sqrt{1+\epsilon}}{\sqrt{1-\epsilon}}$
with probability at least $1-\delta$.
The function $c_2(x)$ is increasing in $[0,1]$, 
and L'H\^{o}pital's rule implies that 
$c_2(\epsilon) \rightarrow 2$ as $\epsilon \rightarrow 0$ and $c_2(\epsilon) \rightarrow 1/(2\ln(2)-1) \leq 2.6$ as $\epsilon \rightarrow 1$.

\paragraph{Uniform sampling, with or without replacement}
The proof is analogous to the corresponding part of the proof 
Theorem~\ref{t_singchernoff}.

\bibliography{RandomizedPapers} 
\end{document}